\documentclass[a4paper]{cas-sc}
\usepackage{amsmath,amssymb,amsthm}
\usepackage[ruled,vlined,linesnumbered]{algorithm2e}
\SetKwRepeat{Do}{do}{while}
\usepackage{kbordermatrix}
\usepackage[numbers,square,sort]{natbib}
\usepackage{url}
\usepackage[nobiblatex]{xurl}
\usepackage{subfigure}
\usepackage{mathtools}
\usepackage{blkarray}
\usepackage{bigstrut} 
\usepackage{multirow}
\usepackage{tikz}
\usepackage{tikz-3dplot}
\usepackage{pgfplots}
\usepackage[shortlabels]{enumitem}
\usepackage{afterpage}
\usepackage{float}
\usepackage{placeins}
\usepackage{pgfplots}
\usepgfplotslibrary{colorbrewer}
\usetikzlibrary{pgfplots.statistics, pgfplots.colorbrewer} 
\usepackage{pgfplotstable}
\usepackage{filecontents}
\usepackage{subcaption}

\usepackage{algorithm2e}
\SetAlFnt{\small}
\SetAlCapFnt{\small}
\SetAlCapNameFnt{\small}
\usepackage{algorithmic}
\algsetup{linenosize=\tiny}

\pgfplotsset{compat=newest}
\usepgfplotslibrary{patchplots,statistics}
\usetikzlibrary{datavisualization}
\usetikzlibrary{datavisualization.formats.functions}

\newtheorem{thm}{Theorem}{\bf}{}
\newtheorem{probl}{Problem}{}{}
{\bf}{}
{\bf}{}
\newtheorem{prop}{Proposition}{\bf}{}
\newtheorem{lem}{Lemma}{\bf}{}
{\bf}{}
{\bf}{}
{\bf}{}
{\bf}{}
\newtheorem{form}{Formulation}{\bf}{}

\DeclareMathOperator*{\argmax}{arg\,max}

\def\tsc#1{\csdef{#1}{\textsc{\lowercase{#1}}\xspace}}
\tsc{WGM}
\tsc{QE}
\tsc{EP}
\tsc{PMS}
\tsc{BEC}
\tsc{DE}

\definecolor{col1}{RGB}{66,49,34}
\definecolor{col2}{RGB}{37,88,105}
\definecolor{col3}{RGB}{220,140,51}
\definecolor{col4}{RGB}{138,138,51}
\definecolor{col5}{RGB}{138,138,51}

\begin{document}
\LinesNumbered
\let\WriteBookmarks\relax
\def\floatpagepagefraction{1}
\def\textpagefraction{.001}
\shorttitle{The PPS on binary, tree-child networks}
\shortauthors{Frohn~\emph{et al.}}

\title [mode = title]{The parental parsimony problem on binary, tree-child phylogenetic networks}

\author[1]{Martin Frohn}[orcid=0000-0002-5002-4049]
\address[1]{Department of Advanced Computing Sciences, Maastricht University, Paul Henri Spaaklaan 1, 6229 EN Maastricht, Netherlands.}
\fnmark[1]
\cormark[1]

\cortext[cor1]{Corresponding author}
\fntext[fn1]{Email: \href{mailto:martin.frohn@maastrichtuniversity.nl}{martin.frohn@maastrichtuniversity.nl} (Martin Frohn)}

\begin{abstract}
Phylogenetic reconstruction is one of the major challenges in computational biology. Among existing reconstruction methods for phylogenetic networks, an important subtask emerges in extending a leaf-labelling on a phylogenetic network to determine a most parsimonious tree inside the network. There exist different variants of this subtask depending on the biological model assumptions for which distinct evolutionary phenomena are captured by the network. In this article we assume that next to hybridization or recombination events, also allopolyploidy or incomplete lineage sorting are present. Then, finding the most parsimonious tree inside the network is called the \emph{parental parsimony score problem} (PPS), a NP-hard combinatorial optimization problem. We provide the first constant-factor approximation for the PPS on arbitrary but fixed leaf labels and a class of networks on which the PPS remains NP-hard, namely binary, semi-simplex, tree-child phylogenetic networks. Furthermore, we introduce a novel exact solution algorithm for the PPS on binary, tree-child phylogenetic networks and analyze its performance on simulated data.
\end{abstract}


\maketitle

\section{Introduction}\label{sec:1}
Molecular phylogenetic reconstruction methods infer evolutionary histories from molecular data~\cite{FelsenBook}. For an integer $n\geq 3$, consider a set $\Gamma =\{x_1,x_2,\dots,x_n\}$ representing a molecular sequence (such as a DNA, RNA or codon sequence). For example, $x\in\Gamma$ can take values over the set $\{A,C,G,T\}$ for DNA sequences or $x$ is encoded by one of the 20 amino acids for protein sequences. The elements of~$\Gamma$ are called \emph{taxa}. Then, we can represent the pattern of diversification events throughout evolutionary history that give rise to the set of taxa $\Gamma$ by an ordered triplet $(T,\phi,w)$, called a \emph{phylogenetic tree} of $\Gamma$, such that $T=(V,E)$ is a rooted tree having $n$ leaves, $\phi$ is a bijection between the leaves of $T$ and the taxa in~$\Gamma$, and $w$ is a vector of non-negative weights associated to the edges of $T$. The latter vector $w$ represents the quantity of evolutionary change along a given edge. We make our study independent of biological assumptions on the set of taxa by choosing an integer $p\geq 1$ such that all taxa in $\Gamma$ take values over an alphabet $S_p=\{0,1,\dots,p\}$. In particular, we assume that our molecular data is given by a function $C:\Gamma\to S_p$ which is called a \emph{character} of~$\Gamma$. We call $S_p$ the set of \emph{character states} and, for $s,t\in S_p$, define the \emph{Hamming distance} $d_H$ of $s$ and $t$ by $d_H(s,t)=1$ if $s\neq t$ and $d_H(s,t)=0$ otherwise. Then, for a phylogenetic tree $(T,\phi,w)$ of $\Gamma$ with $T=(V,E)$ and an \emph{extension} $C'$ of $C$ on $V$ given by a map $C':V\to S_p$ with $C'(\phi^{-1}(x))=C(x)$ for all $x\in\Gamma$, that minimizes
\begin{align*}
\text{score}(T,C')=\sum_{(u,v)\in E}d_H(C'(u),C'(v)),
\end{align*}
we call $\text{score}(T,C')$ the \emph{parsimony score} of $T$ and $C'$. For example, the tree $T=(V,E)$ on the left in Figure~\ref{intro0} shows an extension $C'$ of the character $C:\Gamma\to S_2$ with $C(x_1)=C(x_2)=0$, $C(x_4)=1$, $C(x_3)=C(x_5)=C(x_6)=2$ on $V$ with parsimony score three.
\begin{figure}[pos=!t,align=\centering]
\centering
\includegraphics[scale=0.37]{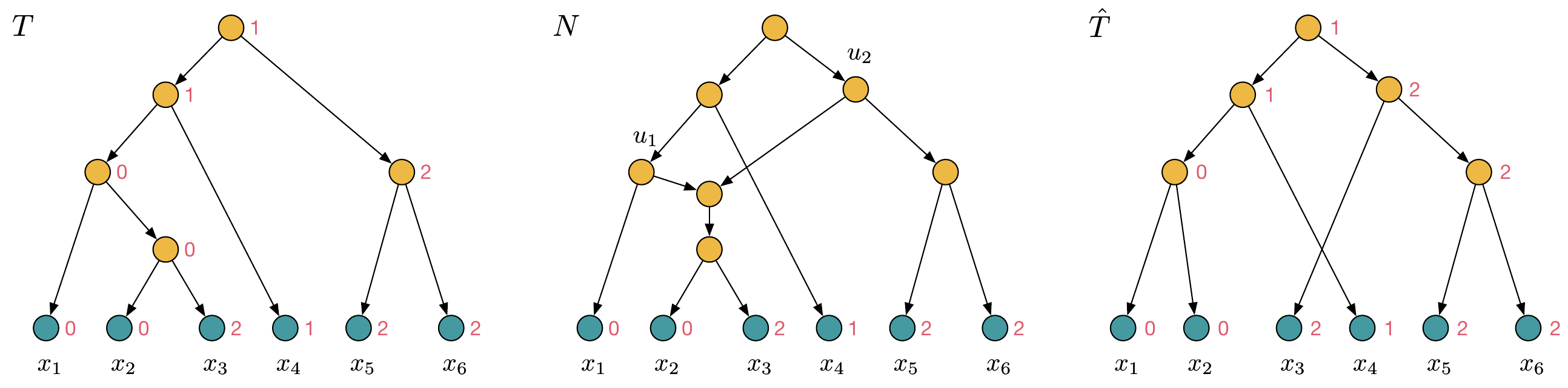}
\caption{On the left a phylogenetic tree $T$ with suppressed edge weights. In red, a character $C$ of $\Gamma$ and one of its extensions $C'$. score$(T,C')=3$ is provably minimum among all extensions of $C$. In the middle a phylogenetic network $N$ such that $T$ and $C'$ constitute a feasible solution to the SPS for $N$ and $C$. On the right a phylogenetic tree $\hat{T}$ parentally displayed by $N$ and an extension $C'$ of $C$ with score$(\hat{T},C')=2$.}\label{intro0}
\end{figure}
A \emph{phylogenetic network} of $\Gamma$ is a generalization of a phylogenetic tree $(T,\phi,w)$ to an ordered triplet $(N,\phi,w)$ where $N$ is a connected directed acyclic rooted graph with $n$ vertices of in-degree 1 and out-degree~0. Hereinafter, given the existence of the bijection $\phi$, with a little abuse of notation, we use the terms phylogenetic tree~$T$ and network~$N$ (and a same symbol) to indicate both a phylogenetic tree and network and the associated graph, respectively. 

To analyse the parsimony score of trees embedded in a phylogenetic network of $\Gamma$, \citet{nakhleh05} introduced the following optimization problem: let $N$ be a phylogenetic network of $\Gamma$, let $C$ be a character of $\Gamma$ and let $\mathcal{T}$ be a family of phylogenetic trees of $\Gamma$ which share some graph structure with $N$. Then, the \emph{Small Parsimony Network Problem}~(SPN) asks for a tree $T\in\mathcal{T}$ and an extension $C'$ of $C$ on the vertex set of $T$ such that the parsimony score of $T$ and $C'$ is minimum.
The SPN is often called the small parsimony problem in the literature and varies in the definition of its input graph~$N$~\citep{nguyen07}. Hence, we can view the SPN as part of a family of parsimony problems which are constrained to a fixed input graph and character. 

In this article we are interested in finding a most parsimonious tree inside a phylogenetic network $N$ of $\Gamma$ with respect to one character $C$ of $\Gamma$. Such a tree minimizes the parsimony score among all phylogenetic trees "displayed" by $N$. However, a definition of displayed trees depends on our biological model assumptions. For example, \citet{jin06} call a phylogenetic tree $T$ of $\Gamma$ displayed by a phylogenetic network $N$ of $\Gamma$ if $T$ is isomorphic to a subtree of $N$ up to edge subdivisions such that the isomorphism is the identity map on $\Gamma$. Under this definition of displayed trees, $T$ in Figure~\ref{intro0} is displayed by network $N$ in the same figure. This means, all descendants of a hybridization vertex~$v$, i.e., a vertex with more than one parent $u$ in the network, have a shared ancestor $u$ in the displayed tree. Then, we arrive at a special case of the SPN~\citep{nakhleh05,fischer15}, called the Softwired Parsimony Score Problem~(SPS) which requires $\mathcal{T}$ to be the set of displayed trees of $N$.

The limitation of the SPS to displayed trees of a network does exclude evolutionary phenomena like allopolyploidy~\citep{warren12} or incomplete lineage sorting (ILS)~\citep{yu13} which may occur at the time of hybrid speciations. Both events posit that more than one parent $u$ of a hybridization vertex $v$ form common ancestors for descendants of $v$. For example, taxa $x_2$ and $x_3$ in the network $N$ from Figure~\ref{intro0} might inherit genes from ancestors $u_1$ and $u_2$, respectively. Hence, the tree $\hat{T}$ in the same figure is not displayed by $N$ but might be more parsimonious than any displayed tree of $N$ with respect to $C$. Indeed, the parsimony score of $\hat{T}$ and $C'$ on the right in the same figure is smaller than the minimum parsimony score of a tree displayed by $N$ and an extension of $C$ on the vertex set of the displayed tree. Hence, generalizing the notion of a displayed tree to allow for allopolyploidy or ILS can aid the development of better phylogenetic reconstruction inference methods~\cite{huber16}. Specifically, we call a phylogenetic tree $T$ of $\Gamma$ \emph{parentally displayed} by a phylogenetic network $N$ of $\Gamma$ if $T$ is not necessarily isomorphic to a subgraph of $N$ but preserves directed paths and degree constraints in $N$ (up to edge subdivisions)~\cite{huber06}. Then, we call the SPN which requires $\mathcal{T}$ to be the set of parentally displayed trees of $N$ the Parental Parsimony Score Problem~(PPS)~\cite{van2018improved}.

In general, the complexity of the SPN is determined by the choice of $\mathcal{T}$. For example, if $\mathcal{T}$ contains exactly one phylogenetic tree, then the SPN can be solved in polynomial time~\citep{fitch71}. In contrast, the SPS and PPS are both NP-hard. Specifically, for binary, tree-child phylogenetic networks~\citep{kong1}, the SPS is NP-hard for binary characters~\citep{fischer15}. The same result holds for the PPS even if in addition the network has reticulation depth one~\citep{van2018improved}.

\begin{figure}[pos=!t,align=\centering]
\centering
\includegraphics[scale=0.37]{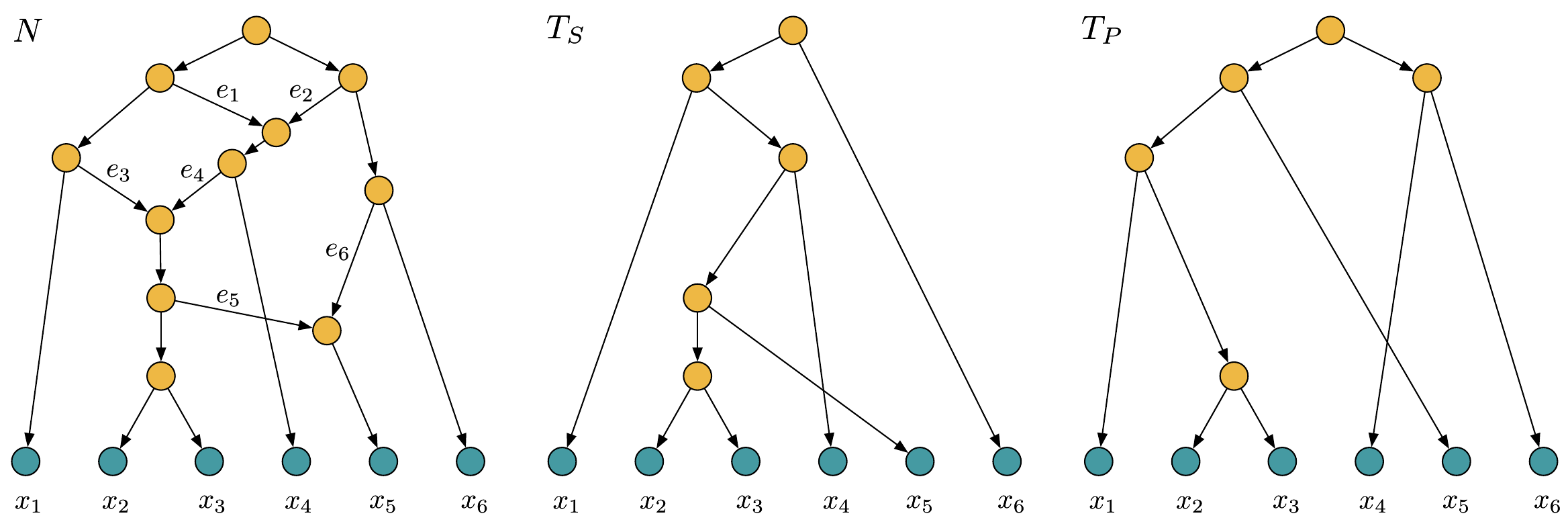}
\caption{Consider the set of taxa $\Gamma =\{x_1,x_2,\dots,x_6\}$. Then, the directed graph on the left shows a phylogenetic network $N$ of $\Gamma$. The graph in the middle (right) shows a phylogenetic tree $T_S$ ($T_P$) of $\Gamma$ which is (parentally) displayed by $N$.}\label{intro}
\end{figure}

While the definitions of displayed trees and parentally displayed trees of phylogenetic networks are similar, e.g., every displayed tree is parentally displayed, the combinatorial structure of the SPS and PPS exhibits clear differences already for small number of taxa. For example, consider the phylogenetic network $N$ on six taxa in Figure~\ref{intro}. Then, the directed path containing edges $e_1$, $e_4$ and $e_5$ is preserved by the tree $T_S$ and $T_P$ in the same figure. The clear difference in the topology of trees~$T_S$ and~$T_P$ arises from the three hybridization vertices which are in a shared directed path in $N$. This difference in the preservation of the same network structure affects the optima of the SPS and PPS on $N$, too. For example, for character $C:\Gamma\to\{0,1\}$ defined by $C(x_1)=C(x_2)=C(x_3)=C(x_5)=0$ and $C(x_4)=C(x_6)=1$, there exist extensions $C'$ and $C''$ of $C$ on the vertices of $T_S$ and $T_P$, respectively, such that score$(T_S,C')=2$ and score$(T_P,C'')=1$. To study the relationship between the SPS and PPS we focus on binary, tree-child phylogenetic networks of $\Gamma$. This subclass is topologically restricted in a specific way which often makes NP-hard optimization problems on phylogenetic networks (comparatively) easier to solve, see e.g. \cite{van2010locating,van2022practical,frohn25}. 

In this article, after preliminaries in Section~\ref{sec:2}, we review the approximation algorithm for the SPS of \citet{frohn25} in Section~\ref{sec:review}. Subsequently, in Section~\ref{sec:approx} we study some connections between the SPS and PPS and extend the approximation algorithm for the SPS to an approximation algorithm for the PPS on binary, tree-child phylogenetic networks. In addition, for binary, semi-simplex, tree-child phylogenetic networks, we show that our method has approximation factor two. In Section~\ref{sec:ilp} we introduce an integer programming formulation for the PPS on binary phylogenetic networks to design an exact solution algorithm for the PPS. Finally, in Section~\ref{sec:exp} we analyze the performance of both our approximate and exact method to solve the PPS on simulated binary, tree-child phylogenetic networks of $\Gamma$ and characters of $\Gamma$. In Section~\ref{sec:conclusion} we conclude this study and discuss future research directions.

\section{Preliminaries}\label{sec:2}
In this section we formalize concepts and introduce notation used throughout the article. For a graph $G$ we denote $V(G)$ and $E(G)$ as the vertex set and edge set of $G$, respectively. We call a connected graph $G$ \emph{rooted} if $G$ is directed and has a unique vertex $\rho\in V(G)$ having in-degree zero and there exists a directed path from $\rho$ to any vertex of in-degree 1 and out-degree 0 of $G$. We call $\rho$ the \emph{root} of $G$ and write $\rho(G)$ whenever $G$ is not clear from the context. In this article we only consider rooted directed acyclic graphs. For $v\in V(G)$, let $G[v]$ be the graph defined as the union of all directed paths in $G$ starting in $v$. Note that $v$ is the root of $G[v]$. We also call $G[v]$ a \emph{pendant} subgraph of $G$. We call a directed graph $G$ \emph{binary} if, for all $v\in V(G)$, the in-degree and out-degree of $v$ sum to at most three. Moreover, we partition $V(G)$ into vertices of out-degree zero $V_{\text{ext}}(G)$, called \emph{external vertices} or \emph{leaves}, and $V_{\text{int}}(G)=V(G)\setminus V_{\text{ext}}(G)$, called \emph{internal vertices}. We omit the mention of $G$ from the definition of these symbols whenever $G$ is clear from the context to simplify our notation. We call vertices in $V_{\text{int}}$ with in-degree at least two \emph{reticulation vertices} or \emph{hybridization vertices} and edges $(u,v)\in E$ for which $v$ is a reticulation vertex \emph{reticulation edges}. We call vertices in $V_{\text{int}}$ with out-degree at least two \emph{tree vertices}. Any pair of vertices in $V$ which have a shared parent vertex and are not each others child are called \emph{siblings}. When $G$ is a phylogenetic network, as introduced in the last section, then we assume there exist no vertices $v\in V$ with in-degree and out-degree one. Moreover, we call the maximum number of reticulation vertices on any directed path in $G$ the \emph{reticulation depth} of $G$. For example, the phylogenetic network $N$ in Figure~\ref{intro} has reticulation depth three. We call a phylogenetic network with reticulation depth one also \emph{semi-simplex}~\cite{steel25}. We call a rooted phylogenetic network \emph{tree-child} if at least one child of each vertex in $V_{\text{int}}$ is a tree vertex or leaf. For example, the network $N$ in Figure~\ref{intro} is tree-child. However, if we remove taxon $x_4$ and the resulting edge subdivision from $N$, then,  for edges $e_1=(u_1,v)$, $e_2=(u_2,v)$, the unique child of $v$ is a reticulation vertex, i.e., the network is no longer tree-child. We denote $\mathcal{N}$ as the set of rooted, binary, tree-child phylogenetic networks of $\Gamma$. For a character $C:\Gamma\to S_p$, a phylogenetic network $N$ of $\Gamma$, a vertex $v\in V(N)$ and the set of taxa $\Gamma'\subseteq\Gamma$ assigned to the leaf set of $N[v]$, we define $C_{|v}:\Gamma'\to S_p$ by $C_{|v}(x)=C(x)$ for all $x\in\Gamma'$ and call $C_{|v}$ the \emph{restriction} of $C$ to $\Gamma'$.

\begin{figure}[pos=!t,align=\centering]
\centering
\includegraphics[scale=0.4]{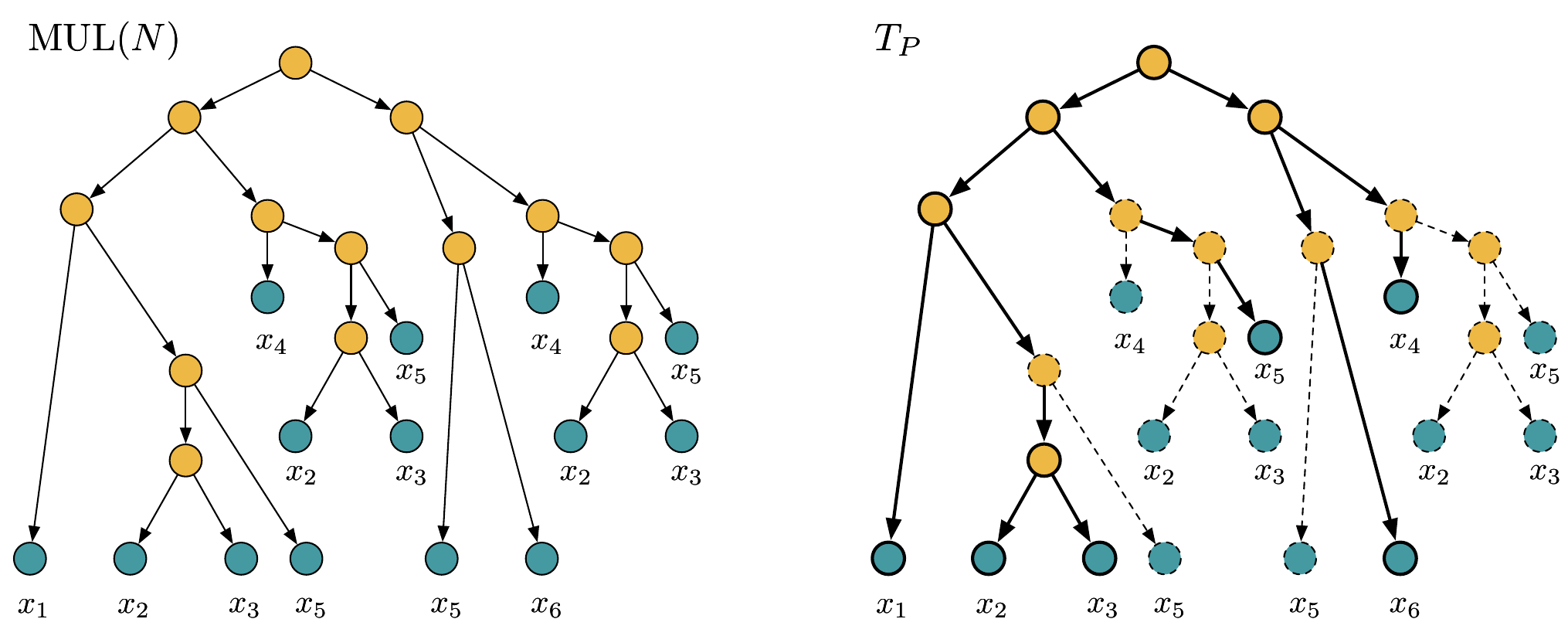}
\caption{The graph on the left is a multilabeled tree $T$ of $\Gamma=\{x_1,\dots,x_6\}$ such that $T=\,\text{MUL}(N)$ for $N$ in Figure~\ref{intro}. The tree $T_P$ on the right side indicated by bold lines is a subgraph of MUL$(N)$ and isomorphic to $T_P$ in Figure~\ref{intro} (up to edge subdivisions).}\label{fig:parentallyDisplayed}
\end{figure}

For a phylogenetic network $N$ of $\Gamma$ and a phylogenetic tree $T$ of $\Gamma'$, analogous to \citet{jin06}, we call $T$ \emph{displayed} by $N$ if and only if $T$ is isomorphic to a subtree of $N$ up to edge subdivisions such that the isomorphism is the identity map on $\Gamma$. To define parentally displayed trees we consider the notion of a multilabeled tree exhibited by a phylogenetic network introduced by \citet{huber06}. Specifically, we call a phylogenetic tree of a multiset of $\Gamma$ a \emph{multilabeled tree} of $\Gamma$. Then, in the notation of \citet{van2018improved} we define MUL$(N)$ as the tree whose vertices are the directed paths in $N$ starting at $\rho(N)$, and whose edges $(P,P')$ (for directed paths $P,P'$ in $N$ starting at $\rho(N)$) exist if and only if $P$ can be extended by one edge in $N$ to obtain $P'$. In addition, MUL$(N)$ is multilabeled such that leaves $P$ in MUL$(N)$ are labeled by the taxon in $\Gamma$ assigned to the endpoint of $P$ in $N$. Then, we call $T$ \emph{parentally displayed} by $N$ if and only if $T$ is a subgraph of MUL$(N)$ up to edge subdivisions. As an example, for $N$ from Figure~\ref{intro}, MUL$(N)$ and a parentally displayed tree $T_P$ of $N$ are shown in Figure~\ref{fig:parentallyDisplayed}. In essence, a parentally displayed tree of $N$ preserves directed paths and degree constraints in $N$ up to edge subdivisions while, in contrast to displayed trees, allows for the duplication of subgraphs. Since MUL$(N)$ has exponential size in the input length, we will not make use of it in our design of a polynomial-time approximation algorithm for the parental parsimony problem. In other words, the only purpose of MUL$(N)$ in this article is to define parentally displayed trees. Now, we are ready to define the SPS and PPS formally:
\begin{probl}
Given an integer $n\geq 3$, a set $\Gamma=\{x_1,x_2,\dots,x_n\}$ of $n$ taxa, a phylogenetic network $N$ of $\Gamma$ and a character $C$ of $\Gamma$, we call
\begin{enumerate}
\item finding a phylogenetic tree $T$ of $\Gamma$ displayed by $N$ and an extension $C'$ of $C$ on $V(T)$ such that score$(T,C')$ is minimum the \emph{Softwired Parsimony Problem} (SPS) of $N$ and $C$.
\item finding a phylogenetic tree $T$ of $\Gamma$ parentally displayed by $N$ and an extension $C'$ of $C$ on $V(T)$ such that score$(T,C')$ is minimum the \emph{Parental Parsimony Problem} (PPS) of $N$ and $C$.
\end{enumerate}
\end{probl}
\noindent We also refer to the SPS/PPS of $N$ and $C$ as the SPS/PPS on $N$ if the choice of $C$ is clear from the context.

The graph we obtain from a phylogenetic network $N$ by deleting exactly one reticulation edge $(u,v)$ for all reticulation vertices $v$ is called a \emph{switching} of $N$~\cite{kelk2019finding} (or \emph{embedded phylogenetic tree} of $\Gamma$ in $N$~\cite{semple2016phylogenetic}). Let $\mathcal{T}(N)$ denote the set of all phylogenetic trees displayed by the phylogenetic network $N$ and $\mathcal{S}(N)$ the set of switchings of the network. An attractive property of tree-child networks is that a phylogenetic tree $T$ is in $\mathcal{T}(N)$ if and only if there is a switching $S\in\mathcal{S}(N)$  that is isomorphic to $T$ up to edge subdivisions (Theorem 1.1 in~\cite{semple2016phylogenetic}). This property holds for tree-child networks $N$ and all $S\in\mathcal{S}(N)$ because for each vertex $v\in V(N)$ or $v\in V(S)$, there exists a directed path from $v$ to a leaf in $N$ such that each internal vertex of that path is a tree vertex. Hence, in this case it is sufficient to determine a switching of $N$ to solve the SPS on $N$. In more general network classes switchings can also be used to characterize displayed trees~\cite{fischer15}, but there the isomorphism up to edge subdivisions does not necessarily hold: the switching might contain unlabeled leaf vertices, for example. 

\begin{figure}[pos=!t,align=\centering]
\centering
\includegraphics[scale=0.4]{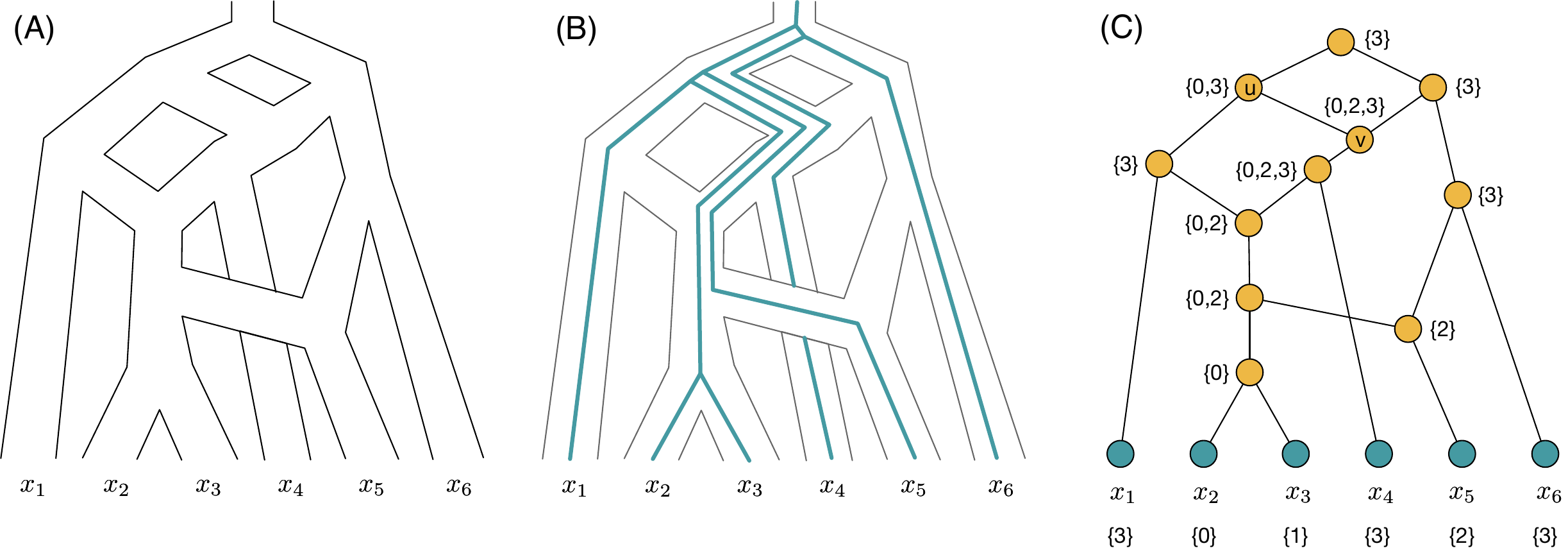}
\caption{Figures (A), (B) and (C) are various visualizations of the phylogenetic network $N$ from Figure~\ref{intro}. In (A) $N$ is depicted like a species tree in coalescence theory~\cite{degnan2009gene}, each branch capturing multiple generations of a species population. Then, (B) can be seen as reconciling a gene tree (shown in blue) with the species tree from (A)~\cite{yu2013fast}. Finally, (C) shows a lineage function $\lambda :V(N)\to 2^{S_3}$ with $\lambda(v)$ labeling all $v\in V(N)$. This lineage function maps the gene tree shown in (B) into $N$. For example, for vertex $u$ in $N$ there are four lineages entering the corresponding speciation event in (B) (going backwards in time as common in coalescence theory) with the multiset of four character states $\{0,2,3,3\}$ coming from the children of $u$. Hence, $\lambda(u)=\{0,3\}$ captures the two character states associated with the two lineages leaving $u$.}\label{fig:lineage}
\end{figure}

Since our primary interest in this article is the PPS for rooted, binary, tree-child phylogenetic networks $\mathcal{N}$ and characters of~$\Gamma$, we consider a formalization of the PPS that is only valid for rooted binary phylogenetic networks based on the notion of lineage functions. \citet{van2018improved} showed that the PPS for binary phylogenetic networks $N$ of $\Gamma$ and a character $C:\Gamma\to S_p$ can be written as a minimization problem over the class of set functions $\lambda:V\to 2^{S_p}$, called \emph{(p-width) lineage functions} on $N$. Figure~\ref{fig:lineage} illustrates the intuition for the purpose of lineage functions. For a lineage function~$\lambda$, \citet{van2018improved} defined weights $w_\lambda(v)$ for all $v\in V$:
\begin{align*}
w_\lambda(v)=\begin{cases}0 &\text{if $v=\rho$},\\
\left|\lambda(v)\setminus\bigcup\limits_{(u,v)\in E}\lambda(u)\right| &\text{if $v\neq\rho$ and $|\lambda(v)|\leq\sum\limits_{(u,v)\in E}|\lambda(u)|$},\\
\infty &\text{otherwise,}
\end{cases}
\end{align*}
For example, for the lineage function in Figure~\ref{fig:lineage}(C), all weights $w_{\lambda}(v)$, $v\in V(N)$, are finite and the only non-zero weights are $w_{\lambda}(u)=w_{\lambda}(v)=w_{\lambda}(x_3)=1$ (where $x_3$ denotes the leaf assigned taxon $x_3$). Then, the authors proved the following formulation of the PPS:
\begin{prop}\label{prop::linfunc}
Let $N$ be a rooted, binary phylogenetic network of $\Gamma$ and let $C:\Gamma\to S_p$ be a character. Then, the following formulation of the PPS for $N$ and $C$ is valid:
\begin{align*}
\min\left\{\sum\limits_{v\in V(N)}w_{\lambda}(v)\,:\,\lambda\text{ is a $p$-width lineage function on $N$},~~\lambda(x)=\{C(x)\}~~\forall\,x\in\Gamma,~~|\lambda(\rho(N))|=1\right\}.
\end{align*}
\end{prop}

Notice that Proposition~\ref{prop::linfunc} characterizes the objective function value of the PPS but not trees parentally displayed by~$N$. For example, for lineage function $\lambda$ in Figure~\ref{fig:lineage}(C) we have $\sum_{v\in V(N)}w_{\lambda}(v)=3$, i.e., the parsimony score of some tree $T$ parentally displayed by $N$ and character $C:\{x_1,\dots,x_6\}\to S_3$ with $C(x_2)=0$, $C(x_3)=1$, $C(x_5)=2$, $C(x_1)=C(x_4)=C(x_6)=3$ is three. The topology of $T$ is unknown from this perspective.

\begin{figure}[pos=!t,align=\centering]
\centering
\includegraphics[scale=0.38]{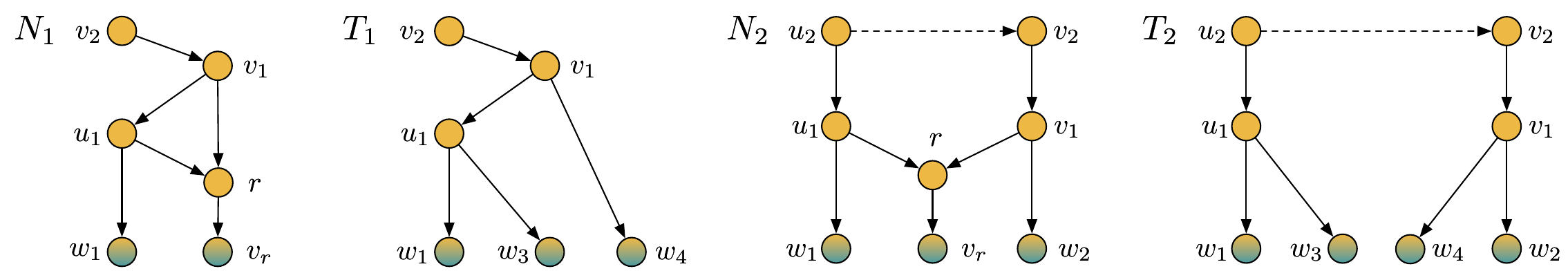}
\caption{For a rooted, binary, tree-child network $N$, the figure contains two possible subgraphs $N_1$ and $N_2$ of $N$ depicting parents and the child of a reticulation vertex $r$ and their adjacency relations. Dashed edges might not exist and we allow $u_2=v_2$. Bicolored vertices can be internal vertices or leaves. If the vertices $v_r$ are roots of pendant subtrees~$T$ of~$N$ and $w_3,w_4\in V(T[v_r])$ such that $V_{\text{ext}}(T[w_3])\cap V_{\text{ext}}(T[w_4])\neq\emptyset$ and $V_{\text{ext}}(T[w_3])\cup V_{\text{ext}}(T[w_4])=V_{\text{ext}}(T[v_r])$, then $T_1$ and $T_2$ can be extended to trees parentally displayed by $N$.}\label{treechild}
\end{figure}

Lastly we want to make use of the tree-child property when studying the PPS. To this end, the local structure of reticulation vertices in a binary, tree-child network is shown in Figure~\ref{treechild}. Indeed, binary, tree-child networks are characterized by two forbidden subgraphs, namely two reticulation vertices can neither be adjacent nor siblings (Theorem 1.1 in \cite{semple2016phylogenetic}). This means, the adjacency relations of the parents and child of a reticulation vertex $r$ have to be of the form $N_1$ or $N_2$ in Figure~\ref{treechild}. Hence, we refer to the two subnetworks $N_1$ and $N_2$ in Figure~\ref{treechild} as the \emph{reticulation graphs}. With a slight abuse of notation we say trees $T_1$ and $T_2$ in Figure~\ref{treechild} are parentally displayed by $N_1$ and $N_2$ in the same figure, respectively, if $v_r$ is not a leaf in some phylogenetic network $N$ containing $N_1$ or $N_2$ as subgraphs. In this case we assume there exist directed paths from $v_r$ to both $w_3$ and $w_4$ in $N$, i.e., $T_1$ and $T_2$ are displayed by MUL$(N_1)$ and MUL$(N_2)$, respectively. 

\section{A review of the state-of-the-art polynomial-time approximation algorithm for the softwired parsimony problem on binary, tree-child phylogenetic networks}\label{sec:review}
In this section we survey the polynomial-time 2-approximation algorithm for the SPS of~\citet{frohn25} 
as our first building block to solve the PPS for $N\in\mathcal{N}$ and a character $C$ of $\Gamma$. The key idea for the development of the authors' approximation algorithm is to analyze the integer linear program for the SPS introduced by~\citet{fischer15} within the framework of primal-dual methods. Here, we will not review the authors' methodology (see~\cite{abdulaziz2026formal} for a primer on primal-dual methods) but observe how the authors apply the methodology to approximate the SPS, resulting in Algorithm~\ref{SPSapprox}. Crucially, the results on the SPS for $N\in\mathcal{N}$ will allow us to study the PPS for semi-simplex $N\in\mathcal{N}$ as a purely local optimization problem, simplifying our analysis in the next sections.

First, we consider the starting point of~\citet{frohn25}'s study: linear programming duality of \citet{fischer15}'s formulation of the SPS for phylogenetic trees. To this end, for a rooted, binary phylogenetic tree of $\Gamma$ and a character of $\Gamma$, let $\lambda$ be lineage function on $T$ with $\lambda(x)=\{C(x)\}$ for all $x\in\Gamma$ with corresponding binary decision variables $z_v^s$ for all $v\in V_{\text{int}}$, $s\in S_p$, such that
\begin{align*}
z_v^s&=\begin{cases}1&\text{if $s\in\lambda(v)$},\\
0&\text{otherwise},
\end{cases}
\end{align*}
Then, the following program is an IP formulation of the SPS for $T$ and $C$~\cite{fischer15}:
\begin{form}\label{form::btps}
\begin{align}
\min~~\sum_{e\in E(T)}c_e~~~~~~~~~&\label{primal:obj}\\
\text{s.t.}\,~~~~~~c_e-z_u^s+z_v^s&\geq 0~&~&\forall\,e=(u,v)\in E(T),s\in S_p\label{primal:con1}\\
c_e+z_u^s-z_v^s&\geq 0~&~&\forall\,e=(u,v)\in E(T),s\in S_p\\
\sum_{s=0}^pz_v^s&= 1~&~&\forall\,v\in V(T)\\
z_v^{C(v)}&=1~&~&\forall\,v\in V_{\text{ext}}\\
c_e&\in\{0,1\}~&~&\forall\,e\in E(T)\\
z_v^s&\in\{0,1\}~&~&\forall\,v\in V(T),s\in S_p
\end{align}
\end{form}
\noindent One can observe that variables $c_{(a,b)}$ in Formulation~\ref{form::btps} encode the Hamming distance between $d_H(C'(a),C'(b))$ for extension $C'$ of $C$ from $\Gamma$ to $V(T)$ corresponding to binary variables $z_v^s$ with $\sum_{s\in S_p}z_v^s=1$. In other words, the objective function~\eqref{primal:obj} coincides with the parsimony score of $T$ and $C'$. Next, consider the dual of the LP relaxation of Formulation~\ref{form::btps}:
\begin{form}\label{form::btpsDual}
\begin{align}
\max~~\sum_{v\in V(T)}\mu_v+\sum_{v\in V_{\text{ext}}}\eta_v\,~~~~~~~~~~~~~~~~~~~~~~~~~~~~~~~~~~~~~~~~~~~~~~~~~~~~~~~~~~~~~&\\
\text{s.t.}~~~~~~~~~~~~~~~~~~~~~~~~~~~~~~~~~~~~~~~~~~~~~~~~~~~~~~~~~~~~~~~~~~~~~~\sum_{s=0}^p\left(\beta_{u,v}^s+\beta_{v,u}^s\right)&\leq 1~&~&\forall\,(u,v)\in E(T)\\
\mu_v+\sum_{(v,w)\in E(T)}\left(\beta_{w,v}^s-\beta_{v,w}^s\right)+\sum_{(u,v)\in E(T)}\left(\beta_{u,v}^s-\beta_{v,u}^s\right)&\leq 0~&~&\forall\,v\in V_{\text{int}},s\in S_p\label{con1::mu1}\\
\mu_v+\sum_{(u,v)\in E(T)}\left(\beta_{u,v}^s-\beta_{v,u}^s\right)&\leq 0~&~&\forall\,v\in V_{\text{ext}},s\in S_p,\,s\neq C(v)\label{con1::mu2}\\
\mu_v+\eta_v+\sum_{(u,v)\in E(T)}\left(\beta_{u,v}^s-\beta_{v,u}^s\right)&\leq 0~&~&\forall\,v\in V_{\text{ext}},s=C(v)\label{con1::mu+eta}\\
\beta_{u,v}^s,\beta_{v,u}^s&\geq 0~&~&\forall\,(u,v)\in E(T),s\in S_p\\
\mu_v&\in\mathbb{R}~&~&\forall\,v\in V(T)\\
\eta_v&\in\mathbb{R}~&~&\forall\,v\in V_{\text{ext}}
\end{align}
\end{form}
\noindent We write $(c,z)$ and $(\beta,\mu,v)$ as shorthand notation for feasible solutions to Formulations~\ref{form::btps} and~\ref{form::btpsDual}. \citet{frohn25} investigate the primal-dual complementary slackness conditions of Formulations~\ref{form::btps} and~\ref{form::btpsDual} to prove sufficient conditions for optimality of $(c,z)$ and $(\beta,\mu,v)$. For example, constraints~\eqref{primal:con1} on edges $(u,v)\in E(T)$ and character states $s\in S_p$ correspond to dual variables $\beta_{u,v}^s$. Hence, feasible $(c,z)$ and $(\beta,\mu,v)$ are optimal if complementary slackness conditions
\begin{align*}
\left(c_{(u,v)}-z_u^s+z_v^s\right)\beta_{u,v}^s&=0~&~&\forall\,(u,v)\in E(T),\,s\in S_p
\end{align*}
hold. Then, as an illustration we observe that the complementary slackness conditions with respect to $v\in V_{\text{int}}$ hold for a particular choice of dual variables (part of Proposition 2 in~\cite{frohn25}):

\begin{prop}\label{prop::complementary}
Let $T$ be a phylogenetic tree of $\Gamma$ and let $C$ be a character of $\Gamma$. Let $(c,z)$ and $(\beta,\mu,\nu)$ be feasible solutions of Formulation~\ref{form::btps} and~\ref{form::btpsDual}, respectively. Let $\lambda$ be the lineage function corresponding to $z$ and let $(u,v),(v,w_1),(v,w_2)\in E(T)$, $w_1\neq w_2$. Then, $(c,z)$ and $(\beta,\mu,\nu)$ satisfy all complementary slackness conditions with respect to $v$
\begin{enumerate}
\item if $\lambda(v)=\lambda(w_1)\cap \lambda(w_2)$ and one of the following two sets of conditions holds:
\begin{enumerate}[(i)]
\item $\lambda(u)\neq \lambda(v)=\lambda(w_1)=\lambda(w_2)=s$ and $\beta_{v,u}^s=1$, $\beta_{v,u}^t=0$, $t\neq s$, $\beta_{u,v}^t=0$, $t\in S_p$, $\mu_v=-1$;
\item $\lambda(u)=\lambda(v)=\lambda(w_1)=\lambda(w_2)$ and $\beta_{v,u}^s=1/|\lambda(v)|$, $s\in \lambda(v)$, $\beta_{v,u}^t=0$, $t\in S_p\setminus\lambda(v)$, $\beta_{u,v}^t=0$, $t\in S_p$, $\mu_v=-1/|\lambda(v)|$.
\end{enumerate}
\item if $\lambda(v)\subseteq\lambda(w_1)\cup\lambda(w_2)$, $\lambda(v)\cap\lambda(w_1)\neq\emptyset$, $\lambda(v)\cap\lambda(w_1)\cap\lambda(w_2)=\emptyset$, $|\lambda(w_1)|=|\lambda(w_2)|$ and $\beta_{v,u}^s=1/|\lambda(v)|$, $\mu_v=1/|\lambda(v)|-1/|\lambda(w_1)|$.
\end{enumerate}
\end{prop}
While the sufficient conditions for complementary slackness in Proposition~\ref{prop::complementary} hold for a seemingly arbitrary\linebreak parametrization of feasible solutions $(c,z)$ and $(\beta,\mu,\nu)$, they serve as a useful bijection between feasible primal and dual solutions. Indeed, Proposition~2 and Corollary~1 in~\cite{frohn25} extend Proposition~\ref{prop::complementary} to a collection of certificates for local optimality and non-optimality which prove optimality for a primal-dual method to solve the SPS for $T$ and $C$ (Proposition~3 in~\citet{frohn25}). In the following we first discuss the actual method and then return to its important feature of achieving local optimality through complementary slackness which we seek to extend to the PPS.

\begin{algorithm}[!t]
 \KwIn{A phylogenetic tree $T$ of $\Gamma$; a character~$C$ of $\Gamma$}
 \KwOut{A solution to the SPS for $T$ and $C$}
 \textbf{for} $v\in V_{\text{ext}}(T)$ \textbf{do} $\lambda(v)\leftarrow\{C(v)\}$\;
 \textbf{for} $v\in V_{\text{int}}(T)$ \textbf{do} $\lambda(v)\leftarrow\{\emptyset\}$\;
   Let $Q$ be a queue and add the parents of leaves of $T$ to $Q$ in a fixed order such that children of vertices $v\in Q$ have a higher priority than $v$\;
  \Do{$Q\neq\emptyset$}{
  	$v\leftarrow\,\text{pop}(Q)$\;
	$w_1,w_2\leftarrow$ children of $v$\;
	\eIf{$\lambda(w_1)\cap \lambda(w_2)\neq\emptyset$}{
		$\lambda(v)\leftarrow \lambda(w_1)\cap \lambda(w_2)$\;
		\textbf{if} \emph{$v=\rho(T)$} \textbf{then} make $\lambda(v)$ a singleton set by removing character states at random\;
		\textbf{if} \emph{$\lambda(v)$ is a singleton set} \textbf{then} resolve$(T[v],\lambda)$\;
	}{
		$\lambda(v)\leftarrow \lambda(w_1)\cup \lambda(w_2)$\;
	}
	Add the parent of $v$ to $Q$\;
  }
  \textbf{for} $v\in V(T)$ \textbf{do} define extension $C':V(T)\to\Gamma$ of $C$ by $\lambda(v)=\{C'(v)\}$\;
\Return{$(T,C')$}\;
\caption{exact solution algorithm for the SPS on phylogenetic trees}\label{SPSapproxTree}
\end{algorithm}

The primal-dual method of~\citet{frohn25} emulates \citet{fitch71}'s algorithm to solve the SPS on phylogenetic trees. To make the definition of \citet{fitch71}'s algorithm amendable for our further study in this article we consider the following operation: for a phylogenetic tree $T$ and a lineage function $\lambda:V(T)\to 2^{S_p}$, we define the operation resolve$(T,\lambda)$ as the traversal of $T[v]$ in preorder while, for $(a,b)\in E(T[v])$, setting $\lambda(b)=\lambda(a)$ if $\lambda(a)\subseteq \lambda(b)$, and making $\lambda(b)$ a singleton set by removing character states at random otherwise. Then, Algorithm~\ref{SPSapproxTree} defines \citet{fitch71}'s algorithm. To illustrate the method consider the tree on the left in Figure~\ref{fig:fitch}. Then, the intermediate state of $\lambda$ in the middle of the same figure arises from $\lambda(1)=\lambda(x_2)\cup \lambda(x_3)=\{0,2\}$, $\lambda(2)=\lambda(x_5)\cup \lambda(x_6)=\{2\}$ (line~12) and subsequently $\lambda(3)=\lambda(x_1)\cap \lambda(1)=\{0\}$ (line~8). Next, the intermediate state of $\lambda$ on the right in Figure~\ref{fig:fitch} is achieved by recognizing that $\lambda(3)$ is a singleton and therefore resolve$(T[3],\lambda)$ is called to set $\lambda(1)=\lambda(3)$ because $\lambda(3)\subset \lambda(1)$ (line~10). Again, line~12 in Algorithm~\ref{SPSapproxTree} determines $\lambda(4)$ and $\lambda(5)$. Finally, lines~9 and~10 can produce $\lambda$ as shown on the left in Figure~\ref{intro0}. 

\begin{figure}[pos=!b,align=\centering]
\centering
\includegraphics[scale=0.37]{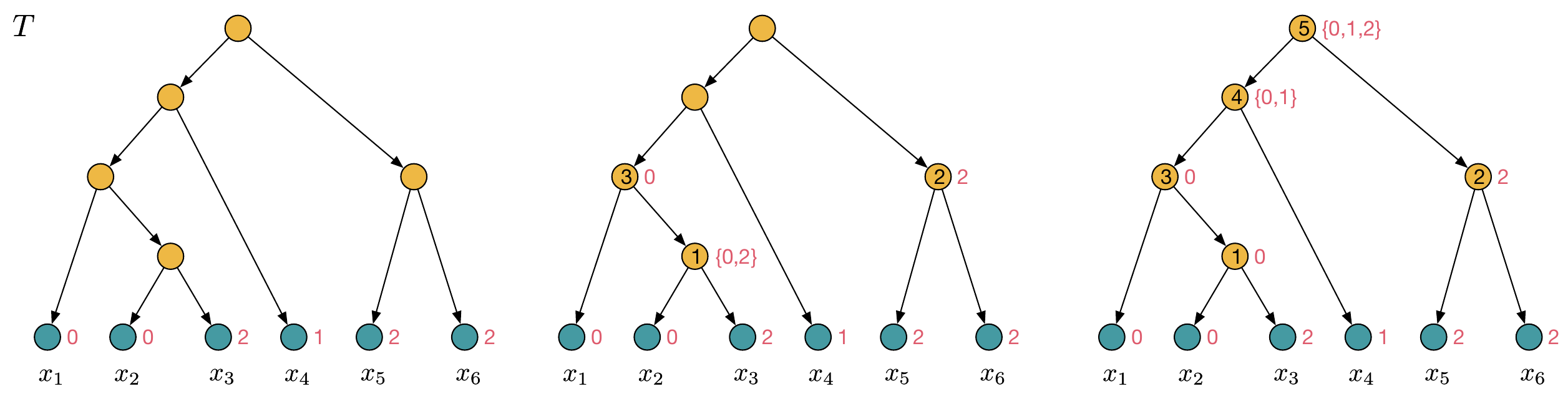}
\caption{On the left a phylogenetic tree $T$ of $\Gamma=\{x_1,\dots,x_6\}$ with a character $C$ of $\Gamma$ in red. For all trees in this figure, leaves are labeled by their assigned taxa and internal vertices are (partially) labeled by the order in queue $Q$ in Algorithm~\ref{SPSapproxTree}. The trees in the middle and on the right depict two different intermediate states of $\lambda$ (in red) in the application of Algorithm~\ref{SPSapproxTree}.}\label{fig:fitch}
\end{figure}

While the definition of Algorithm~\ref{SPSapproxTree} is entirely combinatorial, Proposition~3 in~\cite{frohn25} shows that a bijection between primal and dual variables hinted at in Proposition~\ref{prop::complementary} proves that the assignments to $\lambda(v)$ in lines~8 and~12 of Algorithm~\ref{SPSapproxTree} satisfy the complementary slackness conditions of Formulation~\ref{form::btps} and~\ref{form::btpsDual} with respect to vertex $v$. In the remainder of this section we discuss how the translation by Algorithm~\ref{SPSapproxTree} of local optimality to a global optimum to the SPS for $T$ is achievable for networks in $\mathcal{N}$, too.

A significant simplification of the SPS on phylogenetic trees compared to the general SPS becomes apparent in lines~3 and~13 of Algorithm~\ref{SPSapproxTree}: the method is correct for a simple ordering of vertices in queue $Q$. Intuitively, this invariant holds because for every vertex $v\in V(T)$ full information about all optimal extensions $C'$ of $C_{|v}$ to $V(T[v])$ is stored in line~8 or~12. Recall that the intermediate lineage functions $\lambda$ map vertices to character state sets, i.e., they store information about optimal solutions to the SPS but do not necessarily give rise to extensions of $C_{|v}$ to $V(T[v])$ like in line~15. We do not have this local optimality property for general phylogenetic networks because reticulation edges can connect subgraphs of the network with vastly different topologies. Therefore, \citet{frohn25} study the SPS on rooted, binary, tree-child phylogenetic networks with a more sophisticated queue in mind. Specifically, for $N\in\mathcal{N}$ and an algorithm~$\mathcal{A}$ processing vertices $V_{\text{int}}(N)$, we define a \emph{tree-priority queue} $Q$ as a priority queue such that tree vertices have a higher priority than reticulation vertices $r$ that are ordered based on their induced reticulation graphs: if the sibling(s) of $r$ in Figure~\ref{treechild} have been processed by $\mathcal{A}$, then $r$ has a higher priority than reticulation vertices for which there exists a corresponding unprocessed sibling. Furthermore, to simplify our notation adding reticulation vertices with two unprocessed siblings to $Q$ has no effect on the state of $Q$. Ties between the priorities of vertices in $Q$ are broken by their insertion order. This means, after all tree vertices in $Q$ have been processed like in Algorithm~\ref{SPSapproxTree} and $Q$ is not empty, a reticulation vertex $v$ can be processed in a similar way. For brevity of the discussion we assume without loss of generality that $v$ has two siblings (see Observation~1 in \cite{frohn25}) and we do not discuss the aspect of time-consistency of a phylogenetic network, a property which \citet{frohn25} integrate into their method through the distinction between unprocessed and processed siblings of a reticulation vertex.

\begin{algorithm}[!t]
 \KwIn{Reticulation graph $G$ whose reticulation vertex has two siblings; a positive integer $p$; a map $\lambda:V_{\text{ext}}(G)\to 2^{S_p}$}
 \KwOut{A subgraph of $G$ and a lineage function $\lambda:V(G)\to 2^{S_p}$}
  Let $v$ be the reticulation vertex in $G$, $u_1$ and $u_2$ its parents and $w_1$ its child\;
  Let $w_2$ and $w_3$ be the siblings of $v$ with parents $u_1$ and $u_2$, respectively, such that $|\lambda(w_1)\cap \lambda(w_2)|\geq|\lambda(w_1)\cap \lambda(w_3)|$\;
  \eIf{$\lambda(w_3)\neq\{\emptyset\}$ and $|\lambda(w_1)\cap \lambda(w_2)|>|\lambda(w_1)\cap \lambda(w_3)|$}{
  	Let $u_2'$ be the parent of $u_2$\;
	$V(G)\leftarrow V(G)\setminus\{v,u_2\}$\;
	$E(G)\leftarrow E(G)\cup\{(u_1,w_1),(u_2',w_3)\}\setminus\{(u_1,v),(u_2,v),(v,w_1),(u_2',u_2),(u_2,w_3)\}$\;
	\textbf{if} $\lambda(w_1)\subset \lambda(w_2)$ \textbf{then} $\lambda(w_2)\leftarrow \lambda(w_1)$ and resolve$(T[w_2],\lambda)$\;
	\textbf{if} $\lambda(w_1)\supset \lambda(w_2)$ \textbf{then} $\lambda(w_1)\leftarrow \lambda(w_2)$ and resolve$(T[w_1],\lambda)$\;
		$\lambda(u_1)\leftarrow \lambda(w_1)$\;
	}{
		$\lambda(u_1)\leftarrow \lambda(w_1)\cup \lambda(w_2)$\; 
	}
	\textbf{if} $\lambda(w_3)\neq\{\emptyset\}$ and $|\lambda(w_1)\cap \lambda(w_2)|=|\lambda(w_1)\cap \lambda(w_3)|$ \textbf{then} $\lambda(u_2)\leftarrow \lambda(w_1)\cup \lambda(w_3)$\;
\Return{$G,\lambda$}\;
\caption{process reticulation graph}\label{SPSprocessN2}
\end{algorithm}

\begin{figure}[pos=!b,align=\centering]
\centering
\includegraphics[scale=0.35]{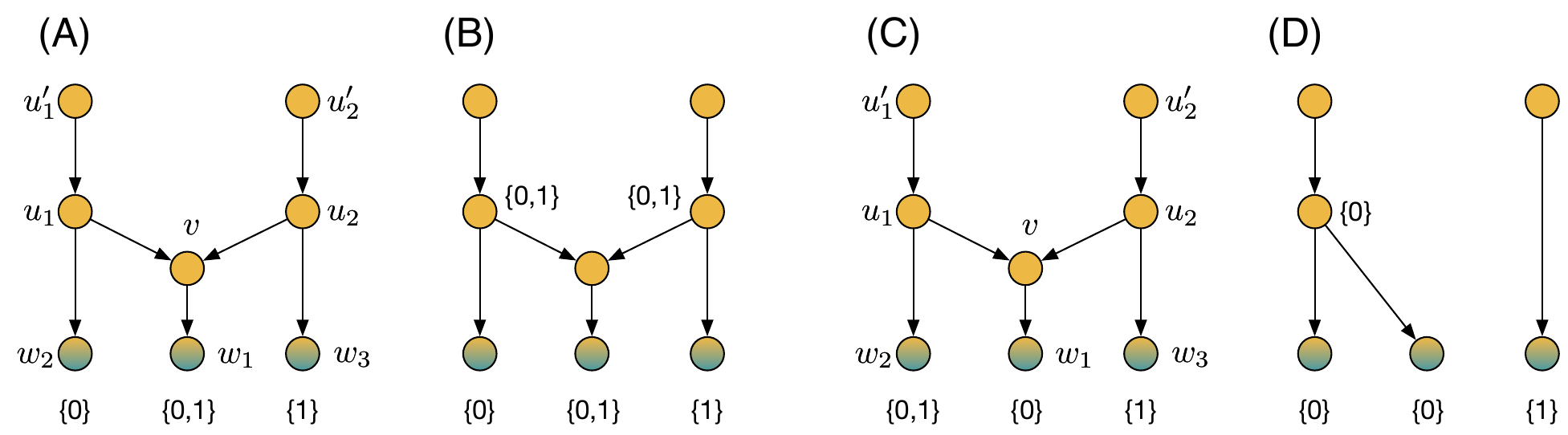}
\caption{The leaves of graphs $G$ in (A) to (D) are labeled by a map $\lambda:V_{\text{ext}}(G)\to 2^{S_p}$. The graphs in (A) and (C) are reticulation graph of type $N_2$ (see Figure~\ref{treechild}).}\label{fig:processN2}
\end{figure}

Then, Algorithm~\ref{SPSprocessN2} encapsulates a propagation of character state sets (as defined by the input map $\lambda$) from the leaves of a reticulation graph $G$ to $v$ and its siblings. For example, for reticulation graph $G$ and map $\lambda$ in Figure~\ref{fig:processN2}(A), we have $|\lambda(w_1)\cap \lambda(w_2)|=|\{0\}|=|\{1\}|=|\lambda(w_1)\cap \lambda(w_3)|$. Hence, by applying Algorithm~\ref{SPSprocessN2} we can process $G$ such that $\lambda(u_1)=\lambda(w_1)\cup \lambda(w_2)=\{0,1\}$, $\lambda(u_2)=\lambda(w_1)\cup \lambda(w_3)=\{0,1\}$ and leaving $G$ unchanged (see Figure~\ref{fig:processN2}(B)). In contrast, if we change the map $\lambda$ slightly to the definition shown in Figure~\ref{fig:processN2}(C), then $|\lambda(w_1)\cap \lambda(w_2)|=|\{0\}|>|\emptyset|=|\lambda(w_1)\cap \lambda(w_3)|$. Hence, Algorithm~\ref{SPSprocessN2} modifies $G$ to obtain the graph in Figure~\ref{fig:processN2}(D). In addition, since $\lambda(w_1)\subset \lambda(w_2)$, we get $\lambda(w_2)=\{0\}$. The intuition behind the two different propagation strategies shown in Figures~\ref{fig:processN2}(B) and (D) is the role of character state set $\lambda(w_1)$ in an underlying network containing $G$ as its subgraph. On the one hand, if $\lambda(w_1)$ agrees with the character state sets of its siblings, then the reticulation vertex $v$ does not provide information about optima to the SPS because removing either one reticulation edge incident to $v$ is locally optimal with respect to the parsimony score. On the other hand, if $\lambda(w_1)$ disagrees with the character state sets of at least one of it siblings, then removing a reticulation edge incident to the shared parent is locally optimal. A proof of these two assertions can be found in the proof of Proposition 4 in \cite{frohn25}. Crucially, Proposition 4 in~\cite{frohn25} is an extension of Proposition 2 in~\cite{frohn25} (see Proposition~\ref{prop::complementary}) proven with the same strategy: first consider \citet{fischer15}'s IP formulation of the SPS for rooted, binary, tree-child phylogenetic networks (generalizing Formulation~\ref{form::btps}) and its LP dual. Then, find sufficient conditions under which complementary slackness holds for a feasible pair of primal and dual solutions to prove that local optimality yields global optimality (page~141 in~\cite{frohn25}):

\SetAlgoSkip{0pt}

\begin{algorithm}[!t]
 \KwIn{A phylogenetic network $N\in\mathcal{N}$; a character~$C$ of $\Gamma$}
 \KwOut{A solution to the SPS for $N$ and $C$}
 \textbf{for} $v\in V_{\text{ext}}(N)$ \textbf{do} $\lambda(v)\leftarrow\{C(v)\}$\;
 \textbf{for} $v\in V_{\text{int}}(N)$ \textbf{do} $\lambda(v)\leftarrow\{\emptyset\}$\;
 Let $T$ be a copy of $N$, let $Q$ be a tree-priority queue and add the parents of leaves of $N$ to $Q$\;
  \Do{$Q\neq\emptyset$}{
  	$v\leftarrow\,\text{pop}(Q)$\;
	\eIf{$v$ is a tree vertex}{
		Apply lines~6 to~13 of Algorithm~\ref{SPSapproxTree}\;
	}{
		\eIf{$v$ has two siblings}{
			process the reticulation graph of $v$ by applying Algorithm~\ref{SPSprocessN2}, modifying $T$ and $\lambda$ according to its outputs $G$ and $\lambda$, respectively\;	
		}{
			Let $u_1$ and $u_2$ be the parents of $v$ and let $w_1$ be the child of $v$\;
			Let $u_2'\neq u_1$ be the parent of $u_2$ and let $w_2$ be the child of $u_1$\;
			$V(T)\leftarrow V(T)\setminus\{v,u_2\}$\;
			$E(T)\leftarrow E(T)\cup\{(u_1,w_1),(u_2',u_1)\}\setminus\{(u_1,v),(u_2,v),(v,w_1),(u_2',u_2),(u_2,u_1)\}$\;
			$\lambda(u_1)\leftarrow \lambda(w_1)\cup \lambda(w_2)$\;
			$\lambda(u_2')\leftarrow \lambda(u_1)$\;
		}
		Add the unprocessed parents and grandparents of $v$ and their unprocessed children to $Q$\;
	}
  }
  \textbf{for} $v\in V(T)$ \textbf{do} define extension $C':V(T)\to\Gamma$ of $C$ by $\lambda(v)=\{C'(v)\}$\;
  For all reticulation vertices $v$ in $T$, remove exactly one reticulation edge $(u,v)$ with $d_H(C'(u),C'(v))$ maximal and remove resulting edge subdivisions\;
  \Return{$(T,C')$}\;
\caption{approximation algorithm for the SPS}\label{SPSapprox}
\end{algorithm}

\begin{prop}\label{prop::SPSapproxOPT}
Let $C$ be a character of $\Gamma$ and let $N$ be a rooted, binary, tree-child network of $\Gamma$. Let $\lambda$ be an intermediate solution in Algorithm~\ref{SPSapprox}. If for all reticulation graphs $N_2$ in $N$ (see Figure~\ref{treechild}) we have $\lambda(w_1)=\lambda(u_2)$ or $\lambda(w_2)=\lambda(v_2)$ or no $N_2$ exists in $N$, then Algorithm~\ref{SPSapprox} extends $\lambda$ to an optimal solution to the SPS for $N$ and $C$.
\end{prop}

\begin{figure}[pos=!t,align=\centering]
\centering
\includegraphics[scale=0.33]{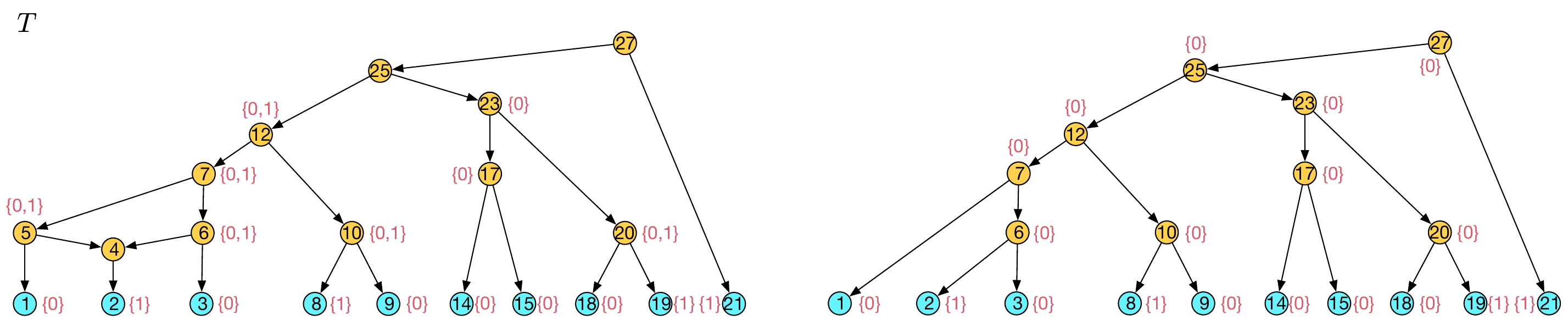}
\caption{Two intermediate solutions $T$ and $\lambda$ of Algorithm~\ref{SPSapprox} for network $N$ in Figure~\ref{example2} without taxa labels and a character of $\Gamma$. For all $v\in V(T)$, the order in which vertices are processed is shown in black and $\lambda(v)$ is shown in red.
}\label{example1}
\end{figure}

\begin{figure}[pos=!b,align=\centering]
\centering
\includegraphics[scale=0.33]{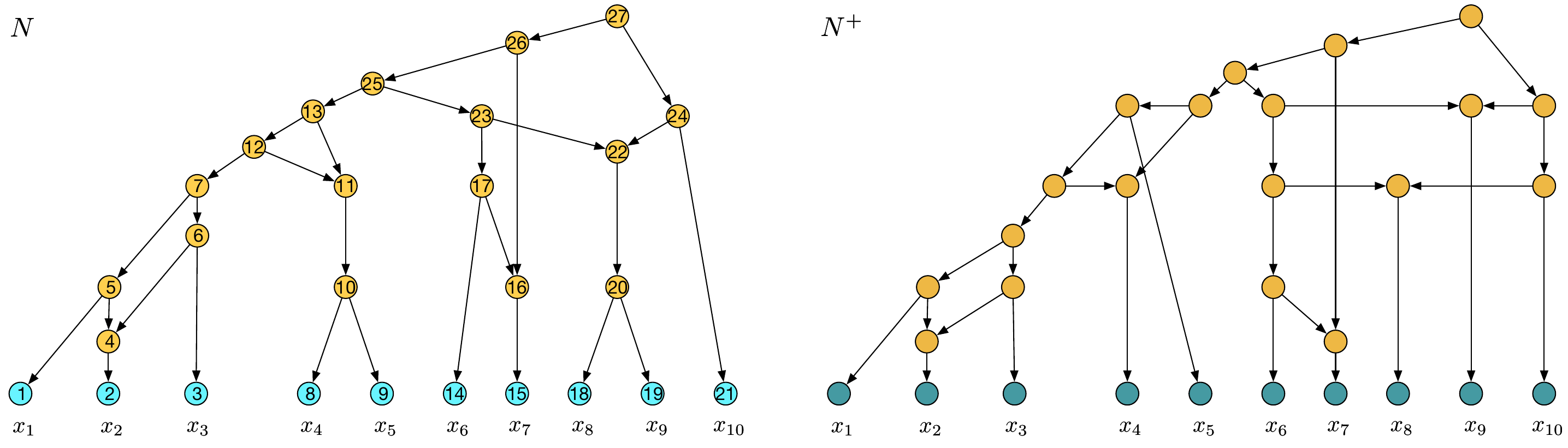}
\caption{The graph on the left is a semi-simplex phylogenetic network $N\in\mathcal{N}$. The graph on the right is a reticulation graph extension $N^+$ of $N$. 
}\label{example2}
\end{figure}

To discuss Algorithm~\ref{SPSapprox}, first consider Algorithm~\ref{SPSprocessN2} which processes reticulation vertices $v$ from queue $Q$ (see line 10 in Algorithm~\ref{SPSapprox}) by changing the topology of a subgraph of $N$ (the reticulation graph $G$ with $v\in V_{\text{int}}(G)$). From Proposition~\ref{prop::SPSapproxOPT} we know that if Algorithm~\ref{SPSapprox} returns a suboptimal solution, then $\lambda(w_1)\neq\lambda(u_2)$ and $\lambda(w_2)\neq\lambda(v_2)$ in reticulation graphs $N_2$ of $N$. This is only possible if neither vertex $u_1$ nor $u_2$ has been removed from $V(G)$ in line~5 of Algorithm~\ref{SPSprocessN2}. Keep in mind that $w_2,w_3,u_2$ in Algorithm~\ref{SPSprocessN2} are labeled $w_1,w_2,v_1$ in Figure~\ref{treechild}, respectively. Furthermore, lines~12 to~17 are an explicit implementation of our assumption that Observation~1 in \cite{frohn25} holds. Hence, the only part of Algorithm~\ref{SPSapprox} we have not discussed so far is line~21 which transform the intermediate network $T$ into a phylogenetic tree. Indeed, line~21 is the shortcoming of this primal-dual method as the analysis by \citet{frohn25} does not permit to control the optimality of the intermediate solution $(T,C')$ with respect to the reticulation edges removed in this line, leading to the result stated in Proposition~\ref{prop::SPSapproxOPT}. In total, we can summarize as follows (tightness follows from Proposition 5 in \cite{frohn25}):

\begin{thm}\label{thm::SPSapprox2}
Let $N\in\mathcal{N}$ and let $C$ be a character of $\Gamma$. Then, Algorithm~\ref{SPSapprox} yields a tight 2-approximation of the SPS for $N$ and $C$ in polynomial time. In addition, if line 21 is not applied to reticulation graphs $N_2$ with $C'(w_1)\neq C'(u_1)$ and $C'(w_2)\neq C'(v_1)$, then Algorithm~\ref{SPSapprox} returns an optimal solution.
\end{thm}

To illustrate Algorithm~\ref{SPSapprox} consider the phylogenetic network $N$ in Figure~\ref{example2} on ten taxa. Then, two intermediate solutions $T$ and $\lambda$ of Algorithm~\ref{SPSapprox} are shown in Figure~\ref{example1} when taking $N$ and the character shown in red on the leaves of $T$ as inputs. The graph on the left shows the intermediate state of $T$ and $\lambda$ right after a reticulation graph was processed for $\lambda(w_3)=\{\emptyset\}$ in Algorithm~\ref{SPSprocessN2}. The graph on the right shows the solution to the SPS returned by Algorithm~\ref{SPSapprox}. Clearly, the output of Algorithm~\ref{SPSapprox} depends on the insertion order of vertices added to $Q$. Since Theorem~\ref{thm::SPSapprox2} applies for any insertion order, we keep the ordering arbitrary but fixed throughout the article.

Importantly, Theorem~\ref{thm::SPSapprox2} does not simply follow from the contrapositive of Proposition~\ref{prop::SPSapproxOPT}. In fact, its approximation guarantee holds for any set of local propagation rules resembling Algorithm~\ref{SPSapprox} (pages 141 and 142 in~\cite{frohn25}). To formalize this point, let $\mathcal{A}$ be an algorithm which propagates a character $C$ of $\Gamma$ from the leaves of $N$ to $\rho(N)$. Specifically, $\mathcal{A}$ processes vertices $v$ in order of a tree-priority queue, applies lines~6 to~13 of Algorithm~\ref{SPSapproxTree} if $v$ is a tree vertex and applies lines~9 to~18 of Algorithm~\ref{SPSapprox}, excluding topological changes to $T$, if $v$ is a reticulation vertex. We refer to $\mathcal{A}$ as the \emph{SPS propagation rules} for $N$ and $C$.
\begin{lem}\label{lem::2approx}
Let $N\in\mathcal{N}$ and let $C$ be a character of $\Gamma$. Let $\mathcal{A}$ be an algorithm which derives a displayed tree $T$ of $N$ and an extension of $C$ on $V(T)$ exclusively from SPS propagation rules for $N$ and $C$. Then, $\mathcal{A}$ yields a 2-approximation of the SPS for $N$ and $C$ in polynomial time, even if $\mathcal{A}$ is applied to every reticulation graph of $N$ independently.
\end{lem}
The reasoning behind Lemma~\ref{lem::2approx} coincides with Proposition~\ref{prop::SPSapproxOPT} because the algorithm stated in Lemma~\ref{lem::2approx} is a realization of a bijection between the primal and dual variables of Formulations~\ref{form::btps} and~\ref{form::btpsDual} which gives rise to the 2-approximation guarantee. The if-condition in Lemma~\ref{lem::2approx} is not necessary but highly useful. Indeed, the analysis of~\citet{frohn25} is entirely localized. Hence, the 2-approximation guarantee holds for every individual subgraph of $N$, in particular reticulation graphs. We will make use of this fact extensively in our development of a 2-approximation algorithm for the PPS.

\section{A 2-approximation algorithm for the parental parsimony problem on binary, semi-simplex, tree-child phylogenetic networks}\label{sec:approx}
While we have seen in Figure~\ref{intro} that the combinatorial aspects of the SPS and PPS differ significantly, we will show in this section how we can employ Algorithm~\ref{SPSapprox} to approximate the PPS for semi-simplex $N\in\mathcal{N}$ with an approximation factor of two. As a preliminary step for this analysis we compare the parsimony score of solutions to the SPS with solutions to the PPS through a reticulation graph extension we introduce in the following subsection.\\

\begin{figure}[pos=!b,align=\centering]
\centering
\includegraphics[scale=0.37]{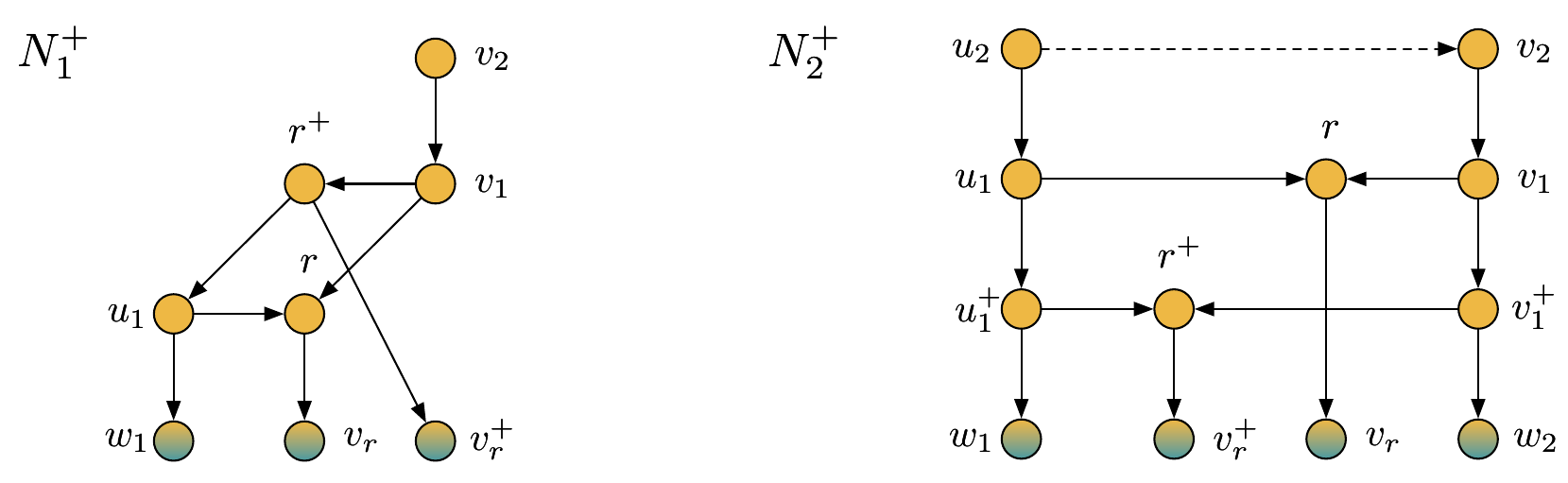}
\caption{The two extensions $N_1^+$ and $N_2^+$ of reticulation graphs $N_1$ and $N_2$ in Figure~\ref{treechild}, respectively.}\label{aux-softwired}
\end{figure}

\subsection{Reticulation graph extensions of binary, semi-simplex, tree-child phylogenetic networks}
We construct auxiliary graphs which map solutions to the SPS for an extended phylogenetic network of $N\in\mathcal{N}$ and a character $C$ of $\Gamma$ to solutions to the PPS for $N$ and $C$. Specifically, for reticulation graphs $N_1$ and $N_2$ (see Figure~\ref{treechild}), we define their \emph{extensions} $N_1^+$ and $N_2^+$, respectively, as shown in Figure~\ref{aux-softwired}. Then, for semi-simplex $N\in\mathcal{N}$, we define a \emph{reticulation graph extension} $N^+$ of $N$ as follows:
\begin{enumerate}[(i)]
\item for reticulation vertices $r\in V(N)$, consider the pendant subgraph $T$ of $N$ with $(r,\rho(T))\in E(N)$. $T$ is a subtree of $N$ because $N$ is semi-simplex. Let $\Gamma_1,\Gamma_2$ be two subsets of the leaves of $T$ such that $\Gamma_1\cup\Gamma_2$ is a bipartition of the leaf set of $T$. If $\Gamma_1,\Gamma_2\neq\emptyset$, then remove $T$ from $N$ except for its root $\rho(T)$ and consider any two rooted binary subtrees $T_1$ and $T_2$ of $T$ with root $\rho(T)$ and leaf set $\Gamma_1$ and $\Gamma_2$, respectively.
\item if $\Gamma_1,\Gamma_2\neq\emptyset$ in (i), then for reticulation graphs $N_i$ of $N$, $i\in\{1,2\}$, replace $N_i$ by its extension $N_i^+$. This means, for $i=1$, adding vertices $r^+,v_r^+$ and edges $(v_1,r^+),(r^+,u_1),(r^+,v_r^+)$, and removing edge $(v_1,u_1)$; for $i=2$, adding vertices $r^+,v_r^+,u_1^+,v_1^+$ and edges $(u_1,u_1^+),(u_1^+,w_1),(u_1^+.r^+),(v_1,v_1^+),(v_1^+,w_2),(v_1^+.r^+),(r^+,v_r^+)$, and removing edges $(u_1,w_1),(v_1,w_2)$.
\item for $T_1,T_2$ with $\Gamma_1,\Gamma_2\neq\emptyset$ and $v_r$ from (i) and (ii), respectively, with $\rho(T_1)=\rho(T_2)=v_r$, identify the vertices $\rho(T_1)$ and $\rho(T_2)$ with vertices $v_r$ and $v_r^+$, respectively.
\end{enumerate}
An example of a reticulation graph extension $N^+$ is shown on the right in Figure~\ref{example2}. Specifically, the reticulation graph associated with vertex $11$ in $N$ is adjacent to the pendant subtree $T[10]$ of $N$. Then, we choose $\Gamma_1=\{x_4\}$ and $\Gamma_2=\{x_5\}$ to partition the leaf set of $T[10]$ in (i). This means, subtrees $T_1$ and $T_2$ of $T[10]$ are constituted only by edge $(10,8)$ and $(10,9)$, respectively. The same arguments apply to reticulation vertex $22$ in $N$ and partition sets $\Gamma_1=\{x_8\}$ and $\Gamma_2=\{x_9\}$. Observe that $N^+$ is not unique because of the free choice of leaf sets $\Gamma_1$ and $\Gamma_2$ in (i) and the two possible identifications in (iii). For the purpose of constructing a displayed tree of $N^+$ in the next subsection we call any graph $P^+$ that we can obtain from $N^+$ by removing some or no reticulation edges and removing some or no subsequent edge subdivisions a \emph{reticulation subgraph extension} of $N$ for $N^+$. For the choice of a reticulation subgraph extension we omit $N$ or $N^+$ in the text when they are clear from the context. The definition of reticulation graph extensions is motivated by the relationship between optimal solutions to the SPS and PPS in Proposition~\ref{prop::PPSapprox2}.

To simplify our notation we say that \emph{$(T^+,C')$ is a solution} to the SPS for $N^+$ and $C$ if $T^+$ is displayed by $N^+$ and~$C'$ is an extension of $C$ on $V(T^+)$. Similarly, we say that \emph{$(T,C')$ is a solution} to the PPS for $N$ and $C$ if $T$ is parentally displayed by $N$ and $C'$ is an extension of $C$ on $V(T)$. We call either solution \emph{optimal} if it minimizes the parsimony score among all solutions to the SPS and PPS, respectively. Let $\text{OPT}_{\text{softwired}}(N^+,C)$ denote the optimal objective function value of the SPS for $N^+$ and $C$, and let $\text{OPT}_{\text{parental}}(N,C)$ denote the optimal objective function value of the PPS for $N$ and $C$. 
\begin{prop}\label{prop::PPSapprox2}
Let $N\in\mathcal{N}$ be semi-simplex and let $C$ be a character of $\Gamma$. Let $N^+$ be a reticulation graph extension of $N$ such that in its construction step
\begin{description}
\item[(i)] we choose $\Gamma_1,\Gamma_2\neq\emptyset$ (and corresponding subtrees $T_1$ and $T_2$) only if there exist $x_1,x_2\in\Gamma$ with $C(x_1)\neq C(x_2)$. In addition, for an integer $k\in\{0,\dots,p\}$, $C(x)=k$ and $C(y)\neq k$ for all $x\in\Gamma_1,y\in\Gamma_2$;
\item[(iii)] we choose $\rho(T_1)=v_r^+$ and $\rho(T_2)=v_r$ if $\Gamma_1,\Gamma_2\neq\emptyset$ in (i).
\end{description}
Let $(T,C')$ be an optimal solution to the PPS for $N$ and $C$. Then, there exist vertices $v\in V(T_1)\cup V(T_2)$ such that integers $k$ in (i) are defined by $C'(v)$. Furthermore, there exists a solution $(T^+,C^+)$ to the SPS for $N^+$ and $C$ such that
\begin{align}\label{lem::UB}
\text{score}(T^+,C^+)\leq\,\text{OPT}_{\text{parental}}(N,C).
\end{align}
In particular, if $C$ is binary, then inequality~\eqref{lem::UB} holds for any choice of $k\in\{0,1\}$ in (i).
\end{prop}
\begin{proof}
Assume $N$ is not a phylogenetic tree. Otherwise our claim holds with equality because parentally displayed trees of $N$ would be displayed trees of $N$. Let $(T,C')$ be an optimal solution to the PPS for $N$ and $C$ and let $N^+$ be a reticulation graph extension of $N$. We prove our claim by showing how to specify $N^+$ to satisfy the construction conditions in the claim statement, transform $T$ into a displayed tree of $N^+$ and extend $C'$ without increasing score$(T,C')$. Since $N$ is not a tree, there exists a reticulation vertex $r\in V(N)$ with associated reticulation graph $N_1$ or $N_2$ (see Figure~\ref{treechild}). Let $\hat{T}_1$ and $\hat{T}_2$ be displayed trees of $N_1$ and $N_2$, respectively. Then, there exists a subgraph of $T$ which is isomorphic to either $\hat{T}_1$, $T_1$, $\hat{T}_2$ or $T_2$ (see Figure~\ref{treechild}) because $T$ is parentally displayed by $N$. Without loss of generality for all subgraphs of $T$ for which such an isomorphism exists we can reinsert a reticulation vertex to obtain a subgraph of $N$. Observe that, for $i\in\{1,2\}$, $\hat{T}_i$ is not isomorphic to any displayed tree of $N_i^+$ in Figure~\ref{aux-softwired} because the leaf sets of $\hat{T}_i$ and $N_i^+$ differ. This means, to transform $T$ into a displayed tree of $N^+$ we need to replace $\hat{T}_i$ by a displayed tree of $N_i^+$ for $i\in\{1,2\}$, respectively. 

First, assume $N^+$ only contains extensions $N_2^+$ and we obtain every subgraph of $T$ isomorphic to $\hat{T}_2$ by removing reticulation edge $(v_1,r)$ and subsequent edge subdivisions in Figure~\ref{treechild}. Then, we transform $T$ into a displayed tree $T^+$ of $N^+$ such that we obtain every subgraph of $T^+$ isomorphic to a displayed tree of $N_2^+$ by removing reticulation edges $(v_1,r)$ and $(v_1^+,r^+)$ and subsequent edge subdivisions in Figure~\ref{aux-softwired}. We choose an extension $C^+:V(T^+)\to S_p$ with $C^+(v)=C'(v)$ for all $v\in V(T)$, set $C^+(u_1^+)=C'(u_1)$ for $u_1,u_1^+\in V(N_2^+)$ and $C^+(v_r^+)=C'(v_r)$ for $v_r,v_r^+\in V(N_2^+)$. Hence, possibly
\begin{align}\label{lem::condition}
\text{score}(T^+,C^+)>\,\text{score}(T,C')~~~~\text{only if }d_H\left(C^+(u_1^+),C^+(v_r^+)\right)=1~~~\text{for }v_r^+,u_1^+\in V(N_2^+).
\end{align}
Therefore, we choose subtrees $S_1$ and $S_2$ of $T[v_r]$, $v_r\in V(N_2)$, in part (i) of the construction of $N^+$ (named $T_1$ and $T_2$ in the described construction, respectively) such that $C(x)=C'(v_r)=\,C^+(v_r^+)$ for all $x\in\Gamma_1$ and $C(y)\neq C'(v_r)=\,C^+(v_r^+)$ for all $y\in\Gamma_2$. Moreover, we identify $\rho(S_1)$ and $\rho(S_2)$ with $v_r^+$ and $v_r$ in part (iii). If $d_H(C^+(u_1^+),C^+(v_r^+))=0$, then our claim follows because of condition~\eqref{lem::condition}. Otherwise $d_H(C^+(u_1^+),C^+(v_r^+))=1$. Then, we apply Algorithm~\ref{SPSapproxTree} for inputs $S_1$ and restriction $C_{|\rho(S_1)}$ with pre-image $\Gamma_1$, and for inputs $S_2$ and restriction $C_{|\rho(S_2)}$ with pre-image $\Gamma_2$, to obtain solutions $(S_1,C_1)$ and $(S_2,C_2)$ to the SPS for the inputs, respectively. For $i\in\{1,2\}$, substitute $C^+(v)$ by $C_i(v)$ for vertices $v\in V(S_i)$. If $d_H(C^+(u_1),C_2(v_r))=0$, i.e., after the substitution $d_H(C^+(u_1),C^+(v_r))=0$, then our claim follows because of $d_H(C'(u_1),C'(v_r))=1$ and the optimality of $(T,C')$. Otherwise $d_H(C^+(u_1),C_2(v_r))=1$ and, for $x\in\Gamma_1$ before the substitution $$1=d_H(C^+(u_1^+),C^+(v_r^+))=d_H(C^+(u_1^+),C(x))=d_H(C^+(u_1^+),C_1(v_r^+)).$$ 
This means, after the substitution $d_H(C^+(u_1),C^+(v_r))=d_H(C^+(u_1^+),C^+(v_r^+))=1$. However, statement~\eqref{lem::condition} no longer holds because score$(T[v_r],C_{|v_r}')>\,\text{score}(S_1,C_1)+\,\text{score}(S_2,C_2)$ which follows from the fact that $C(x)\neq C(y)$ for all $x\in\Gamma_1$, $y\in\Gamma_2$.
Thus, we conclude score$(T^+,C^+)\leq\,\text{score}(T,C')$.

Next, assume $N^+$ only contains extensions $N_1^+$ and we obtain every subgraph of $T$ isomorphic to $\hat{T}_1$ by removing reticulation edge $(v_1,r)$ and subsequent edge subdivisions in Figure~\ref{treechild}. Then, removing reticulation edge $(v_1,r)$ and subsequent edge subdivisions from $N_1^+$ in Figure~\ref{aux-softwired} yields the same subgraph we have considered for the previous isomorphism between a subgraph of $T^+$ and a displayed tree of $N_2^+$. Thus, the same arguments apply and our claim follows.

The cases in which we remove reticulation edge $(u_1,r)$ in $N_1^+$ or $N_2^+$ follow from the same arguments up to a symmetric relabeling. The cases in which a subgraph of $T$ is isomorphic to $T_1$ or $T_2$ are not affected by the reticulation graph extension construction.

Lastly, assume that $C$ is binary. Then, our choice of $S_1$ and $S_2$ can be made independently from $C'$ by keeping $C(x)$ constant for all $x\in\Gamma_1$ and $C(y)\neq C(x)$ for all $x\in\Gamma_1,\,y\in\Gamma_2$ because then, for $C(x)= C'(u_1)$, $x\in\Gamma_1$, our claim still holds. Otherwise, we reverse the role of $S_1$ and $S_2$.
\end{proof}

In subsection~\ref{sec::nonbinary} we will investigate Proposition~\ref{prop::PPSapprox2} to avoid the construction of an optimal solution to the PPS in the case that $C$ is not binary. Notice that the converse of Proposition~\ref{prop::PPSapprox2}, i.e., $\text{OPT}_{\text{softwired}}(N^+,C)>\,\text{score}(T,C')$ for a solution $(T,C)$ to the PPS for $N$ and $C$, does not hold because the character states of $u_1,v_r$ and $u_1^+,v_r^+$ in a displayed tree $T^+$ of $N^+$ restricted to $N_2^+$ can differ when minimizing the parsimony score of $T^+$ across all extensions of $C$ on $V(T^+)$. Hence, without further assumptions on the character states information is lost when transforming character state assignments from $N_2^+$ to $N_2$.

We conclude this subsection by introducing definitions which will allow us to apply Proposition~\ref{prop::PPSapprox2} to intermediate states of variables in our approximation scheme. To this end, for a reticulation graph extension $N^+$ of $N$, a reticulation subgraph extension $P^+$ of $N$ for $N^+$, a displayed tree $T^+$ of $P^+$ and a lineage function $\lambda:V(P^+)\to 2^{S_p}$, we call $(T^+,P^+,\lambda)$ a \emph{partial solution} to the SPS for $N^+$ and $C$. Then, for a partial solution $(T^+,P^+,\lambda)$, we call
\begin{align*}
\text{ps-score}(T^+,P^+,\lambda)=\min\left\{\text{score}(T^+,C')\,:\,C'\text{ is an extension of $C$ on }V(T^+)\text{ with }C'(v)\in \lambda(v)~\forall\,v\in V(T^+)\right\}
\end{align*}
the \emph{partial softwired parsimony score} of $T^+$, $P^+$ and $\lambda$. Finally, using the notation of Figure~\ref{aux-softwired} which labels vertices in $V(N^+)\setminus V(N)$ with $v^+$ for some $v\in V(N)$, we define $\lambda(T^+,P^+,\lambda)$ as the lineage function $\lambda$ on $N$ such that
\begin{align*}
\lambda(v)&=\begin{cases}
\lambda(v^+) &\text{if }v^+\in V(T^+),~v\notin V(T^+),\\
\lambda(v) &\text{if }v^+\notin V(T^+),~v\in V(T^+),\\
\lambda(v^+)\cup \lambda(v) &\text{if }v^+\in V(T^+),~v\in V(T^+)
\end{cases}
\end{align*}
for all tree vertices $v\in V(N)$, $\lambda(r)=\lambda(v_r)$ for all reticulation vertices $r\in V(N)$ and $\lambda(v)=\{\emptyset\}$ if $v,v^+\notin V(T^+)$ and $v$ is not a reticulation vertex in $V(N)$. We call $\lambda(T^+,P^+,\lambda)$ also a (partial) solution to the PPS for $N$ and $C$, and we call
\begin{align*}
\text{pp-score}(T^+,P^+,\lambda)=\sum_{v\in V(N)}w_{\lambda(T^+,P^+,\lambda)}(v)
\end{align*}
the \emph{partial parental parsimony score} of $T^+$, $P^+$ and $\lambda$.

\begin{algorithm}
 \KwData{A reticulation subgraph extension $P^+$  of a semi-simplex phylogenetic network $N\in\mathcal{N}$; a positive integer $p$; a lineage function $\lambda:V(P^+)\to 2^{S_p}$}
 \KwIn{An internal vertex $v$ of $P^+$}
 \KwOut{A truth value}
 Let $I^+$ be a reticulation graph extension such that $P^+$ is a reticulation subgraph extension of $I^+$\;
		\textbf{if} \emph{$v$ is a tree vertex in $P^+$, $\lambda(v)$ is a singleton set and $v^+\notin V(P^+)$} \textbf{then} resolve$\_$parental$(P^+[v],\lambda)$\;
		\uIf{there exists $(u,v)\in E(P^+)$ such that $\lambda(u)$ is a singleton set and $v^+\in V(I^+)$ is a vertex in an extension $N_2^+$ in $I^+$}{
			\uIf{$v$ is a reticulation vertex in $P^+$ and $v^+\in V(P^+)$}{
				For $(u_1,v),(u_2,v),(v,w)\in E(P^+)$, remove $(u_2,v)$ and $(u_1^+,v^+)$ such that\linebreak $|\lambda(u_1)\cap \lambda(w)|+|\lambda(u_2^+)\cap \lambda(w^+)|$ is minimal\;
			}
			\uElseIf{$v$ is a reticulation vertex in $P^+$ and $v^+\notin V(P^+)$}{
				Let $c^+\in V(I^+)$ be the child of $v^+$\;
				For $(u_1,v),(u_2,v),(v,w)\in E(P^+)$, remove $(u_2,v)$ such that\linebreak $|\lambda(u_1)\cap \lambda(w)|+|\lambda(u_2^+)\cap \lambda(c^+)|$ is minimal\;
				\textbf{if} $(u_2^+,c^+)\notin E(P^+)$ \textbf{then} remove an edge with head $c^+$ from $E(P^+)$ and add $(u_2^+,c^+)$ to $E(P^+)$\;
			}
			\ElseIf{$v$ is a tree vertex in $P^+$ and $v^+\in V(P^+)$}{
				Apply lines~4 to~9 with reversed roles for $v$ and $v^+$\;
			}
			\textbf{for} \emph{$(a,v)\in E(N)$} \textbf{do} resolve$\_$one$\_$reticulation$\_$graph$(P^+[a],\lambda)$\;
			\Return{false}\;
		}
		\ElseIf{$v$ is a reticulation vertex in $P^+$}{
		 Let $C'(u)\in\lambda(u)$ and $C'(v)\in\lambda(v)$\;
		 Remove exactly one reticulation edge $(u,v)$ from $E(P^+)$ with $d_H(C'(u),C'(v))$ maximal\;
		 \Return{false}\;
		} 
		\Return{true}\;
\caption{process reticulation subgraph extension}\label{PPSapprox::retGraph}
\end{algorithm}

\begin{algorithm}[!t]
 \KwIn{A semi-simplex phylogenetic network $N\in\mathcal{N}$; a binary character~$C$ of $\Gamma$}
 \KwOut{A solution to the PPS for $N$ and $C$}
 \textbf{for} $v\in V_{\text{ext}}(N)$ \textbf{do} $\lambda(v)\leftarrow \{C(v)\}$\;
 \textbf{for} $v\in V_{\text{int}}(N)$ \textbf{do} $\lambda(v)\leftarrow\{\emptyset\}$\;
Let $P^+$ be the reticulation graph extension of $N$ defined in Proposition~\ref{prop::PPSapprox2}\;
 Let $Q$ be a tree-priority queue and add the parents of leaves of $P^+$ to $Q$\;
  \Do{$Q\neq\emptyset$}{
  	$v\leftarrow\,\text{pop}(Q)$\;
  	\textbf{if} \emph{$v$ is a tree vertex} \textbf{then} apply lines~6 to~13 of Algorithm~\ref{SPSapproxTree} substituting $P^+$ and resolve$\_$parental$(P^+[v],\lambda)$ for $T$ and resolve$(T[v],\lambda)$, respectively\;
	\If{$v$ is a reticulation vertex}{
		Apply lines~9 to~18 of Algorithm~\ref{SPSapprox}, substituting $P^+$ and resolve$\_$parental$(P^+[v],\lambda)$ for $T$ and resolve$(T[v],\lambda)$, respectively, such that (in the notation of Algorithm~\ref{SPSprocessN2})\\
		\textbf{only if} \emph{$u_2\in V(N)$ and $u_2^+\notin V(P^+)$} \textbf{then} vertex $u_2$ is removed from $V(P^+)$, edge $(u_2',w_3)$ is added to $E(P^+)$ and edges $(u_2',u_2),(u_2,w_3)$ are removed from $E(P^+)$\;
	}
  }
\Repeat{Algorithm~\ref{PPSapprox::retGraph} returns true for all $v\in V_{\text{int}}(P^+)$}{
\textbf{for} $v\in V_{\text{int}}(P^+)$ \textbf{do} apply Algorithm~\ref{PPSapprox::retGraph} for variables $P^+$, $\lambda$ and input vertex $v$\;
}
Remove edge subdivisions in $P^+$\;
\textbf{for} $v\in V(P^+)$ \textbf{do} define extension $C':V(P^+)\to\Gamma$ of $C$ by $\lambda(v)=\{C'(v)\}$\;
  \Return{$(P^+,C')$}\;
\caption{approximation algorithm for the PPS}\label{PPSapprox}
\end{algorithm}



\begin{figure}[pos=!b,align=\centering]
\centering
\includegraphics[scale=0.37]{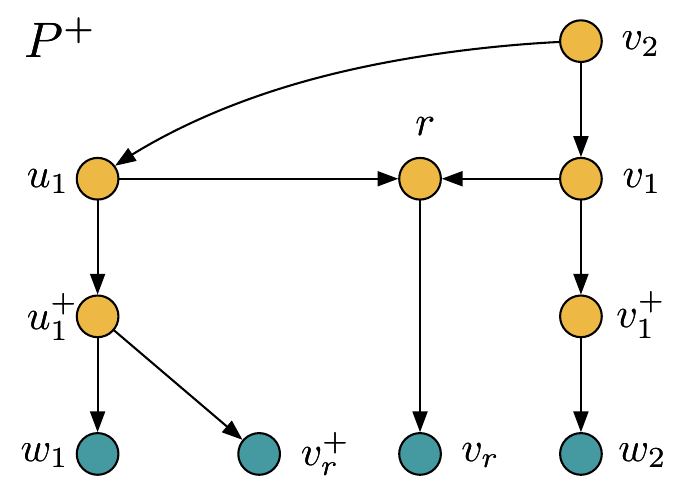}
\caption{A graph obtained from the reticulation graph extension $N_2^+$ in Figure~\ref{aux-softwired} by removing reticulation vertex $r^+$ and incident edges, and adding edge $(u_1^+,v_r^+)$.}\label{fig:retSubgraph}
\end{figure}

\subsection{A 2-approximation algorithm for the parental parsimony problem on binary, semi-simplex, tree-child phylogenetic networks and binary characters}
In this section we investigate the combinatorial similarities between the PPS and SPS to develop a polynomial time 2-approximation algorithm for the PPS for semi-simplex $N\in\mathcal{N}$ and character $C$ of $\Gamma$. First, we consider binary characters $C:V(N)\to S_1$. In this case, Algorithm~\ref{PPSapprox} is applicable (with Algorithm~\ref{PPSapprox::retGraph} written as a subroutine of Algorithm~\ref{PPSapprox} for better readability) and partly resembles Algorithm~\ref{SPSapprox}, drawing parallels between the approximation of the SPS and PPS. However, some crucial steps are different, including a new resolve procedure: for a reticulation subgraph extension $P^+$ of $N$ and a lineage function $\lambda: V(P^+)\to 2^{S_p}$, we define the operation resolve$\_$parental$(P^+,\lambda)$ analogous to $\text{resolve}(P^+,\lambda)$ with the additional constraints that the preorder traversal does not traverse edges $(a,b)\in E(P^+)$ if $b$ is a reticulation vertex of $P^+$ or $b$ is a vertex added in part (ii) of the construction of a reticulation graph extension. This means the traversal stops at $a$ and does not update $\lambda(b)$. For example, for network $P^+$ in Figure~\ref{fig:retSubgraph} the preorder traversal starting in $\rho(P^+)=v_2$ does stop at vertices $u_1$ and $v_1$. Intuitively, this change of the resolve operation allows us to isolate calculations of character states for vertices in the extensions of reticulation graphs. Moreover, we define resolve$\_$one$\_$reticulation$\_$graph$(P^+,\lambda)$ as a relaxed version of operation $\text{resolve$\_$parental}(P^+,\lambda)$ by allowing for the violation of exactly one of the aforementioned additional constraints in the preorder traversal.\\

\begin{prop}\label{equiv::PPS-SPS}
Let $N\in\mathcal{N}$ be semi-simplex and let $C$ be a binary character of~$\Gamma$. Then, Algorithm~\ref{PPSapprox} yields a 2-approximation of the PPS for $N$ and $C$ in polynomial time.
\end{prop}
\begin{proof}
Algorithm~\ref{PPSapprox} constructs a reticulation graph extension $P^+$ of $N$ and subsequently removes reticulation edges and resulting edge subdivisions from $P^+$ until the algorithm terminates. Throughout the whole algorithm variable $P^+$ is a reticulation subgraph extension of $N$. Observe that the construction of $P^+$ in line~3 is well-defined because $C$ is binary. Hence, initially we have
\begin{align}\label{proof::PPS-SPS_OPT}
\text{OPT}_{\text{softwired}}(P^+,C)\leq\,\text{OPT}_{\text{parental}}(N,C).
\end{align}
Lines~5 to~11 in Algorithm~\ref{PPSapprox} are almost identical to lines~4 to~19 in Algorithm~\ref{SPSapprox} and differ only in two aspects. First, resolve$\_$parental restricts the resolution of character state sets of procedure resolve in Algorithm~\ref{SPSapprox} such that the preorder traversal terminates when encountering a reticulation vertex or one of edges $(u_1,u_1^+)$ or $(v_1,v_1^+)$ in extension $N_2^+$ of reticulation graph $N_2$ (see Figures~\ref{aux-softwired} and~\ref{fig:retSubgraph}). The continuation of this truncated traversal is detailed in Algorithm~\ref{PPSapprox::retGraph} which we will discuss later. Secondly, line~10 in Algorithm~\ref{PPSapprox} introduces a necessary condition for the removal of edge subdivisions introduced by line~9, i.e., Algorithm~\ref{SPSprocessN2}: the parent $u_2$ of a reticulation vertex $v$ is kept as an edge subdivision in $P^+$ when reticulation edge $(u_2,v)$ is removed only if $u_2$ is a vertex in a reticulation graph extension $N_2^+$ such that $(u_2,u_2^+)\in E(P^+)$. For example, for a reticulation vertex $r^+$ from reticulation graph extension $N_2^+$, applying lines~9 and~10 to the reticulation graph containing $r^+$ might, for a suitable $\lambda$, yield subgraph $P^+$ in Figure~\ref{fig:retSubgraph}. Since $v_1^+\notin V(N)$, the condition in line~10 does not hold and therefore $v_1^+$ is not removed. Since reticulation edges incident to $r$ in $N_2^+$ cannot be removed before a reticulation edge incident to $r^+$ is removed due to the ordering of queue $Q$, we conclude from the condition in line~10 that no edge subdivisions as a result of reticulation edge removals in $N_2^+$ are removed before $Q$ is empty. 

After line~11, we conclude that $\lambda(v)$ is a non-empty character state set for all $v\in V(P^+)$ and $\lambda$ has been propagated and resolved analogously to Algorithm~\ref{SPSapprox} up to the aforementioned truncation. The modifications of $P^+$ are analogous to the modifications of $T$ in Algorithm~\ref{SPSapprox} except for the lack of removal of edge subdivisions in extensions $N_2^+$. 

Next, applying lines~12 to~14 removes one reticulation edge incident to each remaining reticulation vertex in $P^+$ and after line~15 $P^+$ contains no edge subdivisions, i.e., Algorithm~\ref{PPSapprox} terminates with $P^+$ as a displayed tree of the reticulation graph extension constructed in line~3. Since the number of reticulation edges in $P^+$ is monotone decreasing in the number of applications of line~13, it follows from Lemma~\ref{lem::2approx} that Algorithm~\ref{PPSapprox} has a polynomial runtime. Here we use the fact that Algorithm~\ref{PPSapprox} derives $P^+$ exclusively from SPS propagation rules for $N^+$ and $C$. Notice that lines~4 to~11 in Algorithm~\ref{PPSapprox::retGraph} ensure that $P^+$ is parentally displayed by $N$ after termination by excluding induced subtrees of $N_2^+$ containing either edges $(u_1,v_r)$, $(u_1^+,v_r^+)$ or edges $(v_1,v_r)$, $(v_1^+,v_r^+)$. Also, the resolution of $\lambda$ is completed: if a reticulation edge outside of $N_2^+$ is removed, then line~2 continues the previously truncated preorder traversal. Otherwise, the removal of one or two reticulation edges in an extension $N_2^+$ leads to a continuation of the traversal in line~12 truncated at any further occurrence of a reticulation vertex or another extension $N_2^+$. The order of if-statements (lines~4, 6 and~10) ensures that the repeat-loop of Algorithm~\ref{PPSapprox} terminates with $C'$ constructed by applying line~16 in Algorithm~\ref{PPSapprox} being an extension of $C$ on $V(P^+)$. Thus, $(P^+,C')$ is a solution to the PPS for $N$ and $C$ when Algorithm~\ref{PPSapprox} terminates.

Now, we are left to show that the solution $(P^+,C')$ when Algorithm~\ref{PPSapprox} terminates approximates the PPS for $N$ and $C$ with a factor of two. To this end, we consider a displayed tree $T^+$ of the intermediate reticulation subgraph extension $P^+$ at any point of Algorithm~\ref{PPSapprox} such that $T^+$ is parentally displayed by $N$. Then, we seek to apply Lemma~\ref{lem::2approx} to every reticulation graph of $P^+$ independently. Specifically, we prove the following two claims:
\begin{description}
\item[Claim 1:] the inequality
\begin{align*}
\text{pp-score}(T^+,P^+,\lambda)\leq\,\text{ps-score}(T^+,P^+,\lambda)
\end{align*}
holds whenever all reticulation graphs or their extensions are fully processed or unprocessed by Algorithm~\ref{PPSapprox}.
\item[Claim 2:] for the reticulation graph extension $N^+$ of $N$ defined in Proposition~\ref{prop::PPSapprox2}, inequality
\begin{align*}
\text{ps-score}(T^+,P^+,\lambda)\leq 2\cdot\text{OPT}_{\text{softwired}}(N^+,C)
\end{align*}
holds whenever all reticulation graphs or their extensions are fully processed or unprocessed by Algorithm~\ref{PPSapprox}.
\end{description}
Thus, if claims~1 and~2 hold, then we can conclude our proof for solution $(T^+,C')$ of the PPS for $N$ and $C$ and lineage function $\lambda$ when Algorithm~\ref{PPSapprox} terminates:
\begin{align}\label{final::PPS2approx}
\text{pp-score}(T^+,P^+,\lambda)\stackrel{\text{Claim 1}}{\leq}\,\text{ps-score}(T^+,P^+,\lambda)\stackrel{\text{Claim 2}}{\leq} 2\cdot\text{OPT}_{\text{softwired}}(N^+,C)\stackrel{\eqref{proof::PPS-SPS_OPT}}{\leq} 2\cdot\text{OPT}_{\text{parental}}(N,C).
\end{align}

To prove claims~1 and~2, consider a reticulation vertex $v=\,\text{pop}(Q)$ in line~6. Hence, there exists a reticulation graph (extension) $G$ in $P^+$ containing vertex $v$. By definition of tree-priority queue $Q$ the character state sets of the child of $v$ and at least one sibling of $v$ assigned by $\lambda$ are well-defined. In the following case distinction we apply modifications of $G$ by Algorithm~\ref{PPSapprox} in the notation of Figures~\ref{treechild} and~\ref{aux-softwired}.

\begin{description}
\item[Case 1:] $G=N_1$ (see Figure~\ref{treechild}). Then, $\lambda(w_1)$ and $\lambda(v_r)$ are propagated to vertex $v_2$ in $N_1$. Subsequently, for a partial solution $(T^+,P^+,\lambda)$ such that $T^+$ is parentally displayed by $N$, we have
\begin{align}
w_{\lambda(T^+,P^+,\lambda)}(r)&=\left|\lambda(v_r)\setminus\left(\lambda(u_1)\cup \lambda(v_1)\right)\right|\leq\min\left\{\left|\lambda(v_r)\setminus \lambda(u_1)\right|,\left|\lambda(v_r)\setminus \lambda(v_1)\right|\right\}\label{G::Case1}
\end{align}
In addition, when $G$ is fully processed, i.e., there exists a unique $C'(v)\in\lambda(v)$ for all $v\in V(G)$,
\begin{align*}
\min\left\{\left|\lambda(v_r)\setminus \lambda(u_1)\right|,\left|\lambda(v_r)\setminus \lambda(v_1)\right|\right\}=\min\left\{d_H(C'(v_r),C'(u_1)),d_H(C'(v_r),C'(v_1))\right\}.
\end{align*}
Hence, $\text{pp-score}(T^+,P^+,\lambda)\leq\,\text{ps-score}(T^+,P^+,\lambda)$ is an invariant when propagating $\lambda(w_1)$ and $\lambda(v_r)$ to~$v_2$. This means, claim~1 holds when restricted to $G$. Furthermore, if we apply Algorithm~\ref{SPSapprox} instead of Algorithm~\ref{PPSapprox} to propagate $\lambda(w_1)$ and $\lambda(v_r)$ to~$v_2$, then the propagation is the same because both algorithms derive displayed tree exclusively from SPS propagation rules. Hence, from Lemma~\ref{lem::2approx} we know that claim~2 holds when restricted to $G$ because $G$ is a subgraph of $N$.
\item[Case 2:] $G=N_2$ and $G$ is not a subgraph of an extension $N_2^+$ in $P^+$.
\begin{description}
\item[Case 2.1:] $\lambda(w_2)=\{\emptyset\}$ when $r$ is processed. Then, $\lambda(w_1)$ and $\lambda(v_r)$ are propagated to vertex $u_2$ in $N_2$ and no reticulation edge in~$G$ is removed. Hence, $r$ remains a reticulation vertex until it is processed in line~15 of Algorithm~\ref{PPSapprox::retGraph}. Thus, inequality~\eqref{G::Case1} holds and we can draw the same conclusions as in Case~1.
\item[Case 2.2:] $\lambda(w_2)\neq\{\emptyset\}$ when $r$ is processed. Then, either no reticulation edge in~$G$ is removed and arguments analogous to Case 2.1 apply or one reticulation edge in~$G$ is removed and arguments analogous to Case~1 apply.
\end{description}
\item[Case 3:] $G=N_1^+$ (see Figure~\ref{aux-softwired}). Then, exactly one sibling of $r$ is processed before reticulation vertex $r$ is processed and $r$ is processed before $r^+$. Hence, analogous arguments to Case 2.1 apply and therefore one reticulation edge incident to $r$ is removed in line 15 of Algorithm~\ref{PPSapprox::retGraph}. Assume that $T^+$ does not contain the removed reticulation edge. Let $T_1^+$ be the subgraph of $T^+$ induced by vertex set $V(T^+)\cap V(G)$. Observe that $T_1^+$ is isomorphic to tree $T_1$ in Figure~\ref{treechild}. Therefore, in the notation of Figure~\ref{treechild} we conclude that
\begin{align*}
w_{\lambda(T^+,P^+,\lambda)}(r)&=\left|\left(\lambda(w_3)\cup \lambda(w_4)\right)\setminus\left(\lambda(u_1)\cup \lambda(v_1)\right)\right|\\
&\leq\min\left\{\left|\lambda(w_3)\setminus \lambda(u_1)\right|+\left|\lambda(w_4)\setminus \lambda(v_1)\right|,\left|\lambda(w_4)\setminus \lambda(u_1)\right|+\left|\lambda(w_3)\setminus \lambda(v_1)\right|\right\}
\end{align*}
In addition, when $G$ is fully processed, i.e., there exists a unique $C'(v)\in\lambda(v)$ for all $v\in V(G)$,
\begin{align*}
&\min\left\{\left|\lambda(w_3)\setminus \lambda(u_1)\right|+\left|\lambda(w_4)\setminus \lambda(v_1)\right|,\left|\lambda(w_4)\setminus \lambda(u_1)\right|+\left|\lambda(w_3)\setminus \lambda(v_1)\right|\right\}\\
&=\min\left\{d_H(C'(a),C'(u_1))+d_H(C'(b),C'(v_1))\,:\,a,b\in\{w_3,w_4\},\,a\neq b\right\}.
\end{align*}
Hence, $\text{pp-score}(T^+,P^+,\lambda)\leq\,\text{ps-score}(T^+,P^+,\lambda)$ is an invariant when propagating $\lambda(w_1)$, $\lambda(v_r)$ and $\lambda(v_r^+)$ to~$v_2$. Furthermore, if we apply Algorithm~\ref{SPSapprox} instead of Algorithm~\ref{PPSapprox} to propagate $\lambda(w_1)$, $\lambda(v_r)$ and $\lambda(v_r^+)$ to~$v_2$, then the propagation is the same. Hence, from Lemma~\ref{lem::2approx} we know that claim~2 holds when restricted to $G$ because $G$ is of the same form as $N_2$ when omitting vertex $v_r^+$.
\end{description}
Now, the only remaining subgraph $G$ of $P^+$ that we need to rule out as a counterexample to claims~1 and~2 to conclude our proof is $G=N_2^+$. Here we will move beyond our applications of Lemma~\ref{lem::2approx} so far and argue with our choice of $T^+$ being parentally displayed by $N$. Assume $G=N_2^+$ and consider lines~12 to~14 for vertices in $G$. Then, line~3 in Algorithm~\ref{PPSapprox::retGraph} can hold for only one vertex $v\in V(G)$, namely reticulation vertex $r$ or its child $v_r$ if $r$ has been removed before, because the procedure resolve$\_$parental cannot traverse edges $(u_1,u_1^+),(v_1,v_1^+)\in E(G)$. Let $T_2^+$ be the subgraph of $T^+$ induced by vertex set $V(T^+)\cap V(G)$ in lines~12 to~14.
\begin{figure}[pos=!t,align=\centering]
\centering
\includegraphics[scale=0.37]{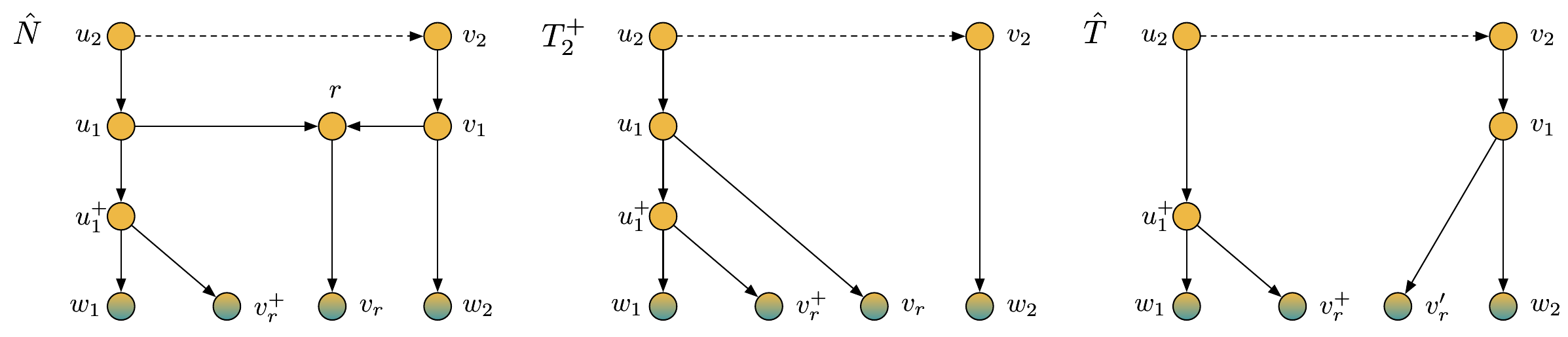}
\caption{An induced subgraph $\hat{N}$ of the reticulation graph extension $N_2^+$ from Figure~\ref{aux-softwired} and two trees $T_2^+$ and $\hat{T}$ displayed by $\hat{N}$.}\label{type2-aux-noIso}
\end{figure}

\begin{description}
\item[Case A:] The subgraph $\hat{N}$ of $P^+$ induced by $V(P^+)\cap V(G)$ contain the same reticulation edges as $T_2^+$, and $T_2^+$ is isomorphic to $T_2$ in Figure~\ref{treechild}. Since $T_2$ is parentally displayed by $N_2$, this case does not affect our choice of $T^+$ as being parentally displayed by $N$. Then, lines~12 to~14 do not affect the optimality of $(P^+,C')$ (when Algorithm~\ref{PPSapprox} terminates) for vertices or edges in $G$ because neither line~3 nor line~14 in Algorithm~\ref{PPSapprox::retGraph} apply. This means, $G$ has been fully processed after two previous applications of Case 2. Therefore we conclude that claim~1 or~2 hold if restricted to $G$. 
\item[Case B:] The subgraph $\hat{N}$ of $P^+$ induced by $V(P^+)\cap V(G)$ contain the same reticulation edges as $T_2^+$, and $T_2^+$ is not isomorphic to $T_2$ in Figure~\ref{treechild}. Without loss of generality $(u_1,v_r),(u_1^+,v_r^+)\in E(T_2^+)$ (see Figure~\ref{type2-aux-noIso}). We deduce that $|\lambda(v_r^+)\cap \lambda(w_1)|>|\lambda(v_r^+)\cap \lambda(w_2)|$ and $|\lambda(v_r)\cap \lambda(u_1^+)|>|\lambda(v_r)\cap \lambda(w_2)|$ because $G$ has already been processed in lines~5 to~11. This means, $\lambda(u_1^+)=\lambda(v_r^+)\cap \lambda(w_1)$ (by applying lines~7 to~9 in Algorithm~\ref{SPSprocessN2}). Since $\lambda(v_r)\cap \lambda(v_r^+)=\emptyset$ by our construction of the reticulation graph extension in line~3, $\lambda(v_r)\cap \lambda(u_1^+)=\emptyset$ follows, contradicting the strict inequality $|\lambda(v_r)\cap \lambda(u_1^+)|>|\lambda(v_r)\cap \lambda(w_2)|$. Thus, this case cannot occur and therefore does not affect our choice of $T^+$ as being parentally displayed by $N$.
\item[Case C:] The edge set of $T_2^+$ contains at least two less reticulation edges than the subgraph $\hat{N}$ of $P^+$ induced by $V(P^+)\cap V(G)$. Among if-conditions in lines~4, 6 and~10 of Algorithm~\ref{PPSapprox::retGraph}, without loss of generality the if-condition in line~6 holds (the if-condition in line~3 holds because $G=N_2^+$). Then, $T_2^+$ is shown as $\hat{N}$ in Figure~\ref{type2-aux-noIso}. 
If applying line~8 in Algorithm~\ref{PPSapprox::retGraph} transforms $\hat{N}$ into $\hat{T}$ in the same figure, then applying line~9 in Algorithm~\ref{PPSapprox::retGraph} does not modify $P^+$. Hence, we can draw the same conclusion as in Case A because $\hat{T}$ is isomorphic to $T_2$ in Figure~\ref{treechild}. However, if applying line~8 in Algorithm~\ref{PPSapprox::retGraph} transforms $\hat{N}$ into $T_2^+$ in Figure~\ref{type2-aux-noIso}, then applying line~9 in Algorithm~\ref{PPSapprox::retGraph} modifies $P^+$ by removing edge $(u_1^+,v_r^+)$ and adding edge $(v_1^+,v_r^+)$ with $v_1^+$ as the edge subdivision of $(v_2,w_2)$, i.e. $T_2^+$ is transformed into $\hat{T}$ if we identify $u_1$ with $u_1^+$ and $v_1^+$ with $v_1$ in $T^+$. The edge addition differentiates this case from all previous cases because ps-score$(T^+,P^+,\lambda)$ possibly increases independently of pp-score$(T^+,P^+,\lambda)$ when applying line~9 of Algorithm~\ref{PPSapprox::retGraph}. We show that this independent increase cannot happen. To this end, assume that applying line~8 of Algorithm~\ref{PPSapprox::retGraph} transforms $\hat{N}$ into $T_2^+$. We investigate how transforming $T_2^+$ into $\hat{T}$ changes ps-score$(T^+,P^+,\lambda)$ when Algorithm~\ref{PPSapprox::retGraph} terminates. 
\begin{description}
\item[Case C.1:] $\lambda(u_1)\neq \lambda(u_2)$. Observe that the preorder traversal resolve$\_$parental visiting $u_2$ stops at vertex $u_1$. Hence, $|\lambda(u_2)|=1$ because $N$ is semi-simplex, i.e., there exists a path containing only tree vertices from the root of $N$ to $u_2$. Moreover, $\lambda(u_1)$ was initially chosen as the union of $\lambda(u_1^+)$ and $\lambda(v_r)$ (see lines~10 to~13 in Algorithm~\ref{SPSprocessN2}).
\begin{description}
\item[Case C.1.1:] $\lambda(u_1)\neq \lambda(v_r)$ if we apply line~12 of Algorithm~\ref{PPSapprox::retGraph} for vertex $r$. Then $\lambda(u_1)=\lambda(u_1^+)$. This means, replacing $T_2^+$ by $\hat{T}$ and setting $\lambda(v_1)=\lambda(v_2)$ cannot increase ps-score$(T^+,P^+,\lambda)$ when Algorithm~\ref{PPSapprox::retGraph} terminates. 
\item[Case C.1.2:] $\lambda(u_1)\neq \lambda(u_1^+)$ if we apply line~12 of Algorithm~\ref{PPSapprox::retGraph} for vertex $r$. Then we can draw the same conclusion as in Case C.1.1 because replacing $(u_2,u_1),(u_1,u_1^+)$ by $(u_2,u_1^+)$ decreases ps-score$(T^+,P^+,\lambda)$ when Algorithm~\ref{PPSapprox::retGraph} terminates. 
\end{description}
\item[Case C.2:] $\lambda(u_1)=\lambda(u_2)$. If $\lambda(u_1)\neq \lambda(v_r)$ when applying line~12 of Algorithm~\ref{PPSapprox::retGraph} for vertex $r$, then arguments analogous to Case~C.1.1 apply. Otherwise, the existence of edge $(u_1,v_r)$ does not affect ps-score$(T^+,P^+,\lambda)$ when Algorithm~\ref{PPSapprox::retGraph} terminates. This means, this case reduces to Case~2.
\end{description}
\end{description}
We conclude that claims~1 and~2 hold. Thus, inequalities~\eqref{final::PPS2approx} hold.
\end{proof}

\begin{algorithm}[!b]
\KwIn{Positive integers $k,l$ with $k\leq l$; a set of taxa $\Gamma$; a rooted, binary phylogenetic tree $T$ of $\Gamma$; a $k$-character~$C$ of $\Gamma$}
\KwOut{An optimal solution to the $(k,l)$-PPS for $T$ and $C$}
\textbf{for} $v\in V_{\text{ext}}(T)$ \textbf{do} $\xi(v)\leftarrow\{C(v)\}$\;
\textbf{for} $v\in V_{\text{int}}(T)$ \textbf{do} $\xi(v)\leftarrow\{\emptyset\}$\;
Let $Q$ be a queue and add the parents of leaves of tree $T$ to $Q$\;
\Do{$Q\neq\emptyset$}{
	$v\leftarrow\,\text{pop}(Q)$\;
	$w_1,w_2\leftarrow$ children of $v$ such that $|\xi(w_1)|\geq|\xi(w_2)|$\;
	\uIf{$|\xi(w_1)|\geq 2$ and there exist $c_1\in\xi(w_1),c_2\in\xi(w_1)$ such that $c_2\subseteq c_1$}{
		$\xi(v)\leftarrow\{c_1\}$\;
	}
	\uElseIf{$|\xi(w_1)|=|\xi(w_2)|=1$ and $|c_1\cup c_2|\leq l$ for $c_1\in\xi(w_1)$, $c_2\in\xi(w_2)$}{
		$\xi(v)\leftarrow\{c_1\cup c_2\}$\;
	}
	\uElse{
		\textbf{for} $c_1\in \xi(w_1),\,c_2\in \xi(w_2)$, $S\subset c_1\cup c_2$ with $|S|=l$ and $c_1\cap c_2\subseteq S$ \textbf{do} $\xi(v)\leftarrow \xi(v)\cup\{S\}$\;
	}
	\If{$v=\rho(T)$}{
		Get $|\xi(v)|=1$ by removing elements from $\xi(v)$ at random\;
	}
	\If{$|\xi(v)|=1$}{
		For $c\in\xi(v)$ and $\alpha_i =\left\{c\cap\text{arg\,max}\left\{|c\cap c'|\,:\,c'\in\xi(w_i)\right\}\right\}$, $i\in\{1,2\}$, set $\xi(w_i)=\alpha_i$ if $\alpha_i\neq\emptyset$ and keep $\xi(w_i)$ unchanged otherwise recursively in preorder\;
	}
	Add the parent of $v$ to $Q$\;
}
\textbf{for} $v\in V(T)$ \textbf{do} $\lambda(v)\leftarrow c$ for the unique set $c\in\xi(v)$\;
\Return{$\lambda$}\;
\caption{exact solution algorithm for the $(k,l)$-PPS}\label{refine-and-cut}
\end{algorithm}

\subsection{Lifting the restriction to binary characters in Algorithm~\ref{PPSapprox}}\label{sec::nonbinary}
In this subsection we investigate Proposition~\ref{prop::PPSapprox2} further to generalize Proposition~\ref{equiv::PPS-SPS} to general characters of $\Gamma$. Notice that Proposition~\ref{equiv::PPS-SPS} does make use of the binary property of the given character $C$ only when Algorithm~\ref{PPSapprox} constructs~$P^+$ uniquely with respect to leaf sets $\Gamma_1$ and $\Gamma_2$ (see lines~3). If $C$ is not binary, then it is unclear which choices for sets $\Gamma_1$ and $\Gamma_2$ do not obstruct optimality in the construction of an approximate solution to the PPS for $N$ and $C$. We formalize this nested subproblem by using the lineage function formulation of the PPS in Proposition~\ref{prop::linfunc}: for a positive integer $k$, we call a function $C:\Gamma\to 2^{S_p}$ with $|C(x_i)|\leq k$ for all $x_i\in\Gamma$ a \emph{$k$-character} of $\Gamma$. Then, for a rooted, binary phylogenetic network $N$ of $\Gamma$, a $k$-character $C$ of $\Gamma$ and an integer $l\geq k$, we call
\begin{align*}
\min\left\{\sum_{v\in V(N)}w_{\lambda}(v)\,:\,\lambda\text{ is a p-width lineage function on }N,~\lambda(x)=C(x)~\forall\,x\in\Gamma,~|\lambda(\rho(N))|\leq l\right\}
\end{align*}
the $(k,l)$-PPS for $N$ and $C$. Observe that for some phylogenetic tree $T$ and some $k$-character $C$ of the leaf set of $T$, the $(1,2)$-PPS for $T$ and $C$ equates to our nested subproblem when defining a reticulation graph extension using $T$. Hence, in the following we show how to solve the $(k,l)$-PPS for $T$ and $C$. Subsequently, we show that Proposition~\ref{equiv::PPS-SPS} holds for all characters of $\Gamma$ by solving the $(1,2)$-PPS to determine bipartitions $\Gamma_1\cup\Gamma_2$ in part (i) of Proposition~\ref{prop::PPSapprox2} without constructing an optimal solution to the PPS for $N$ and $C$. To simplify the notation in Algorithm~\ref{refine-and-cut}, for the set $A_p=2^{S_p}$ of all character state subsets of $S_p$, we introduce functions $\xi:V\to 2^{A_p}$. The purpose of these functions is to capture information about multiple lineage functions $\lambda_1,\dots,\lambda_m$ on $N$ evaluated at a vertex $v\in V(N)$ by defining $\xi(v)=\{\lambda_1(v),\dots,\lambda_m(v)\}$. Figure~\ref{fig::(k,l)-PPSexample} illustrates the use of these sets to solve the $(k,l)-PPS$ with Algorithm~\ref{refine-and-cut}. 
\\

\begin{figure}[pos=!t,align=\centering]
\centering
\includegraphics[scale=0.37]{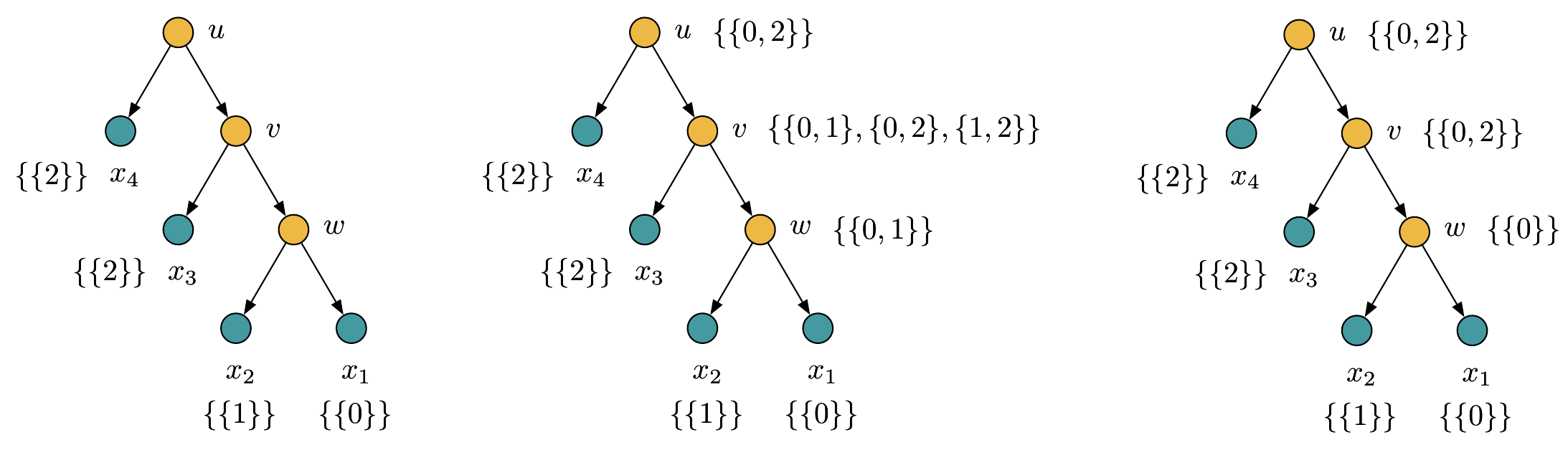}
\caption{The figure on the left shows a rooted, binary phylogenetic tree of $\Gamma=\{x_1,\dots,x_4\}$ and sets $\xi(v)$, $v\in V_{\text{ext}}(T)$ defined by line~1 in Algorithm~\ref{refine-and-cut} for character $C$ of $\Gamma$ given by $C(x_1)=0$, $C(x_2)=1$ and $C(x_3)=C(x_4)=2$. Then, the figure in the middle shows the application of lines~4 to~18 of the same algorithm by recognizing that vertices are removed from queue $Q$ in order $w,v,u$. Note that in each case the if-conditions in lines~13 and~15 do not hold. Finally, the figure on the right depicts sets $\xi(v)$, $v\in V(T)$ after line~18.}\label{fig::(k,l)-PPSexample}
\end{figure}

\begin{prop}\label{prop::kPPS}
Algorithm~\ref{refine-and-cut} is correct and runs in $\mathcal{O}(n^2)$ time.
\end{prop}
\begin{proof}
Assume the inputs of Algorithm~\ref{refine-and-cut} are given. First, lines~1,~2 and~19 ensure that $\lambda$ returned by Algorithm~\ref{refine-and-cut} is a well-defined lineage function. For all $v\in V(T)$, $\xi(v)$ is re-defined at most twice: either in line~8, 10 or 12 and in line~14 or~16. We prove our claim that the output $\lambda$ of Algorithm~\ref{refine-and-cut} is an optimal solution to the $(k,l)$-PPS for $T$ and $C$ by induction on~$n$. The base case is covered by lines~1 and~19. Observe that Algorithm~\ref{refine-and-cut} never constructs a character state set of cardinality larger than $l$. Hence, constraint $|\lambda(\rho(T))|\leq l$ is satisfied. Recall that for $v\in V(N)$ and the set of taxa $\Gamma'\subseteq\Gamma$ assigned to the leaf set of $T[v]$, the function $C_{|v}:\Gamma'\to S_p$ defined by $C_{|v}(x)=C(x)$ for all $x\in\Gamma'$ is the restriction of $C$ to $\Gamma'$. Analogously we consider functions $\xi_{|v}:V(T[v])\to 2^{A_p}$ defined by $\xi_{|v}(w)=\xi(w)$ for all $w\in V(T[v])$. 
Let $\lambda_{|v}$ be the lineage function on $T[v]$ we obtain from $\xi$ when applying lines~13 to~16, and line~19 for $T[v]$ instead of $T$. Then, for children $w_i$, $i\in\{1,2\}$, which we obtain by applying line~6, we assume by our induction hypothesis that Algorithm~\ref{refine-and-cut} is correct for input tree $T[w_i]$ and $k$-character $C_{|w_i}$. In other words, $\lambda_{|w_i}$ is an optimal solution to the $(k,l)$-PPS for $T[w_i]$ and $C_{|w_i}$.
\begin{description}
\item[Case 1:] $|\xi(w_1)|=|\xi(w_2)|=1$ and $|c_1\cup c_2|\leq l$ for $c_1\in\xi(w_1)$, $c_2\in\xi(w_2)$. Then, $\{\lambda_{|w_i}\}=\xi_{|w_i}$, $i\in\{1,2\}$, and $\xi(v)=\{c_1\cup c_2\}$ ensures that $\lambda_{|v}$ is a $l$-character (see line~10). 
\item[Case 2:] $|\xi(w_1)|=|\xi(w_2)|=1$ and $|c_1\cup c_2|> l$ for $c_1\in\xi(w_1)$, $c_2\in\xi(w_2)$. In this case, we consider all subsets $S\subset c_1\cup c_2$ with $|S|=l$ and $c_1\cap c_2\subseteq S$ (see line~12). Since Algorithm~\ref{refine-and-cut} never constructs a character state set of cardinality larger than $l$, at least one set $S$ exists. Furthermore, every set~$S$ minimizes $\sum_{v\in V(T[v])}w_{\lambda}(v)$ among all subsets of $c_1\cup c_2$ of cardinality $l$ because of $\{\lambda_{|w_i}\}=\xi_{|w_i}$ and the optimality of $\lambda_{|w_i}$, $i\in\{1,2\}$.
\item[Case 3:] $|\xi(w_1)|\geq 2$ and $c_2\subseteq c_1\in \xi(w_1)$ for a set $c_2\in\xi(w_2)$. Then, we set $\xi(v)=\{c_1\}$ in line~8 and obtain $\xi(w_1)=\{c_1\}$ and $\xi(w_2)=\{c_2\}$ in line~16, leading to $w_{\lambda}(w_1)=w_{\lambda}(w_2)=0$.
\item[Case 4:] $|\xi(w_1)|\geq 2$ and no set $c_2\in\xi(w_2)$ is a subset of any set in $\xi(w_1)$. Analogous arguments to Case~1 or~2 apply.
\end{description}
In total, by the optimality of $\lambda_{|w_1}$ and $\lambda_{|w_2}$ we conclude that applying lines~13 to~16 to $\xi_{|v}$ for $T[v]$ instead of $T$ yields an optimal solution to the $(k,l)$-PPS of $T[v]$ and $C_{|v}$. Thus, Algorithm~\ref{refine-and-cut} is correct.

To determine the runtime of Algorithm~\ref{refine-and-cut} observe that the cardinality of sets $\xi(v)$ is bounded by a constant because $p$ and $l$ are constant. This means, we can calculate intersections between sets $\xi(v)$ in constant time to implement the if-condition in line~7. Clearly applying the if-condition in line~9 takes constant time because of the cardinality bounds in the condition. Analogously, since there are at most $\binom{p}{l}$ candidates $S$ for the for-loop in line~12, its runtime is bounded by a constant. $\mathcal{O}(n)$ vertices are added to queue $Q$ and removed from $Q$ exactly once. Hence, the total number of vertices visited over all preorder traversals in line~16 is bounded by $\mathcal{O}(nh)$ for the height $h$ of tree $T$. Thus, the total runtime of Algorithm~\ref{refine-and-cut} is $\mathcal{O}(n+nh)=\mathcal{O}(n^2)$.
\end{proof}

Notice that the optimal solution to the $(k,l)$-PPS for $T$ and $C$ returned by Algorithm~\ref{refine-and-cut} is a lineage function $\lambda$ such that for children $w_1,w_2$ of any vertex~$v$ in tree~$T$ we have $\lambda(w_1)\cap\lambda(w_2)\subseteq\lambda(v)\subseteq\lambda(w_1)\cup\lambda(w_2)$. In the following we translate this lineage function into a collection of phylogenetic trees and character extensions: for $m=|\lambda(\rho(T))|$, we define binary phylogenetic trees $T_1,\dots,T_m$ induced by $T$ and extensions $C_i'$ of $C$ on $V(T_i)$, $i\in\{1,\dots,m\}$ as follows: first, we define a map $C':V(T)\to S_p$ to decompose $\lambda$. To this end, set $C'(\rho(T))=s$ for some $s\in\lambda(\rho(T))$, traverse vertices $v\in V(T)$ in preorder and, for $(v,w_1),(v,w_2)\in E(T)$, set $C'(w_i)=C'(v)$ if $C'(v)\in\lambda(w_i)$, $i\in\{1,2\}$. Otherwise, we keep $C'(w_i)$ undefined and do not continue with the traversal for vertex $w_i$. Then, let $T_1$ be the phylogenetic tree induced by vertices $v\in V(T)$ for which $C'(v)$ is well-defined. Let $C_1':V(T_1)\to S_p$ be the extension of $C$ on $V(T_1)$ defined by $C_1'(v)=C'(v)$ for all $v\in V(T_1)$. Furthermore, the lineage function $\mu:V(T)\to 2^{S_p}$ defined by 
\begin{align*}
\mu(v)&=\begin{cases}
\lambda(v)\setminus\{C'(v)\} &\text{if $C'(v)$ is well-defined},\\
\lambda(v) &\text{otherwise}
\end{cases}~&~&\forall\,v\in V(T)
\end{align*}
satisfies
\begin{align*}
\sum_{v\in V(T)}w_{\mu}(v)=\sum_{v\in V(T)}w_{\lambda}(v)-\,\text{score}\left(T_1,C_1'\right).
\end{align*}
This means, we recursively decompose $\mu$ like we did $\lambda$ to define $T_2,\dots,T_m$ and $C_2',\dots,C_m'$ which satisfy
\begin{align*}
\sum_{v\in V(T)}w_{\lambda}(v)=\sum_{i=1}^m\text{score}\left(T_i,C_i'\right).
\end{align*}
We denote split$(T,\lambda)=\{T_i,C_i'\}_{i=1,\dots,m}$.

\begin{thm}\label{prop::approx2depth1}
Let $N\in\mathcal{N}$ be semi-simplex, let $p\geq 1$ be an integer and let $C:\Gamma\to S_p$ be a character of~$\Gamma$. Then, there exists a polynomial time 2-approximation algorithm for the PPS for $N$ and $C$.
\end{thm}
\begin{proof}
Let $v\in V(N)$ be a reticulation vertex. Since $N\in\mathcal{N}$ is semi-simplex, the pendant subgraph $N[v]$ is a tree. Hence, we can apply Algorithm~\ref{refine-and-cut} for $k=1$, $l=2$, $\Gamma$ as the leaf set of $N[v]$, tree $N[v]$ and character $C_{|v}$ to obtain an optimal solution $\lambda_v$ to the $(1,2)$-PPS for $N[v]$ and $C_{|v}$ in polynomial time (see Proposition~\ref{prop::kPPS}). Then, we can write split$(N[v],\lambda_v)=\{T_i,C_i'\}_{i=1,\dots,m}$ for $1\leq m\leq l=2$.

Next, let $N^+$ be the reticulation graph extension of $N$ defined by Proposition~\ref{prop::PPSapprox2}. In the following we will make the construction of $N^+$ explicit without considering an optimal solution to the PPS for $N$ and $C$. let $v\in V(N)$ be a reticulation vertex. If split$(N[v],\lambda_v)=\{T_1,C_1'\}$, then $|\lambda_v(w)|=1$ for all $w\in V(N[v])$ by construction. Hence, there exists no extension of the reticulation graph containing $v$ in $N^+$. Otherwise, split$(N[v],\lambda_v)=\{T_1,C_1',T_2,C_2'\}$. Let $\Gamma_1$ and $\Gamma_2$ be the leaf sets of $T_1$ and $T_2$, respectively. Then, conditions $(i)$ and $(iii)$ in Proposition~\ref{prop::PPSapprox2} for $N^+$ are satisfied by $\Gamma_1$ and $\Gamma_2$ with corresponding subtrees $T_1$ and $T_2$ for some $k\in\{1,\dots,p\}$. The choice of $k$ here is determined by the random removal in line~14 of Algorithm~\ref{refine-and-cut} instead of by character states $C'(v)$, $v\in V(T_1)\cup V(T_2)$, for an optimal solution $(T,C')$ to the PPS for $N$ and $C$ as in Proposition~\ref{prop::PPSapprox2}. Let $P^+$ be the reticulation graph extension that differs from $N^+$ in these choices of integers $k$. Then, we apply Algorithm~\ref{PPSapprox} for construction $P^+$ instead of $N^+$ to obtain a solution $(T^+,C^+)$ to the PPS for $N$ and $C$ (we write $C^+$ instead of $C'$ as denoted in Algorithm~\ref{PPSapprox}). We repeat the whole process for every choice of the random removal in line~14 of Algorithm~\ref{refine-and-cut} for all reticulation vertices in $V(N)$ independently (justified by the use of Lemma~\ref{lem::2approx} in the proof of Proposition~\ref{prop::PPSapprox2}). Hence, we construct $P^+$ and apply Algorithm~\ref{PPSapprox} $\mathcal{O}(n^p)$ times. For at least one of these applications we have $P^+=N^+$, i.e., Proposition~\ref{prop::PPSapprox2} is applicable. Thus, our claim follows from Proposition~\ref{equiv::PPS-SPS} because Proposition~\ref{equiv::PPS-SPS} does not use the binary property of $C$ outside of the application of Proposition~\ref{prop::PPSapprox2}.
\end{proof}

We denote calls to the algorithm in Theorem~\ref{prop::approx2depth1} for $N\in\mathcal{N}$ and character $C$ of $\Gamma$ restricted to one time application of Algorithm~\ref{PPSapprox} as its subroutine by \emph{ApproxPPS$(N,C)$}. Notice that we can ensure that ApproxPPS is well-defined even if $N$ is not semi-simplex, i.e., $N$ has a reticulation depth of at least two, by applying the construction in Proposition~\ref{prop::PPSapprox2} only to every reticulation vertex whose child is the root of a subtree of $N$.

\section{An exact solution algorithm for the parental parmony problem on binary, tree-child phylogenetic networks}\label{sec:ilp}

In this section, we give a branch-\&-bound algorithm for the PPS on rooted, binary, tree-child phylogenetic networks based on a new integer programming formulation of the PPS on rooted, binary networks and an agreement measure between its LP relaxation and algorithm ApproxPPS. To this end, for a rooted, binary phylogenetic network $N$ of $\Gamma$, a lineage function $\lambda$ on $N$, $v\in V(N)$ and $s\in S_p$, we introduce binary decision variables $a_v^s$ such that
\begin{align*}
a_v^s=\begin{cases}
1 &\text{if $s\in\lambda(v)$},\\
0&\text{otherwise},
\end{cases}
\end{align*}
and binary decision variables $c_v^s$ such that $\sum_{s\in S_p}c_v^s=w_{\lambda}(v)$ if $v\neq\rho$ and $w_{\lambda}(v)\neq\infty$. Then, given a character $C$ of $\Gamma$, we consider the following integer program:
\begin{form}\label{form::bpps}
\begin{align}
\min~~\sum_{v\in V_{\text{int}},\,v\neq\rho,\,s\in S_p}c_{v}^s&+\sum_{v\in V_{\text{ext}}}c_{v}\label{obj::BPPS}\\
\text{s.t.}~~~~~~~~~~~~~~~~~~~~~~~~~~c_{v}^s&\geq a_v^s-\sum_{(u,v)\in E}a_u^s~&~&\forall\,v\in V_{\text{int}},\,v\neq\rho,\,s\in S_p\label{bpps::con1}\\
c_v&\geq 1-a_u^{C(v)}~&~&\forall\,(u,v)\in E,\,v\in V_{\text{ext}}\label{bpps::con2}\\
0&\geq \sum_{s\in S_p}\left(a_v^s-\sum_{(u,v)\in E}a_u^s\right)~&~&\forall\,v\in V_{\text{int}},\,v\neq\rho\label{bpps::con3}\\
0&\geq 1-\sum_{s\in S_p}a_u^s~&~&\forall\,(u,v)\in E,\,v\in V_{\text{ext}}\label{bpps::con4}\\
0&\leq 1-\sum_{s\in S_p}a_\rho^s\label{bpps::con5}\\
c_{v}&\in\{0,1\}~&~&\forall\,v\in V_{\text{ext}}\nonumber\\
c_{v}^s&\in\{0,1\}~&~&\forall\,v\in V_{\text{int}},\,v\neq\rho,\,s\in S_p\nonumber\\
a_v^s&\in\{0,1\}~&~&\forall\,v\in V_{\text{int}},\,s\in S_p\nonumber
\end{align}
\end{form}
Constraints~\eqref{bpps::con1} enforce $c_v^s=1$ only if $s\in\lambda(v)$ and $s\notin\lambda(u)$ for all $(u,v)\in E$. Hence, by definition of $w_\lambda(v)$ we ensure that $\sum_{s\in S_p}c_v^s=w_{\lambda}(v)$ if $v\in V_{\text{int}},\,v\neq\rho$ and $w_{\lambda}(v)\neq\infty$. The latter if-condition is established by constraints~\eqref{bpps::con3}. Constraints~\eqref{bpps::con2} and~\eqref{bpps::con4} function analogously to constraints~\eqref{bpps::con1} and~\eqref{bpps::con3}, respectively, by leveraging the fact that we set $\lambda(v)=\{C(v)\}$ for all $v\in V_{\text{ext}}$. Therefore, objective function~\eqref{obj::BPPS} encodes function $\sum_{v\in V}w_\lambda(v)$. Finally, constraint~\eqref{bpps::con5} ensures that $|\lambda(\rho)|=1$. Thus, Formulation~\ref{form::bpps} is valid for the PPS on rooted, binary phylogenetic networks by Proposition~\ref{prop::linfunc}.

\begin{algorithm}[!t]
\KwIn{A phylogenetic network $N\in\mathcal{N}$; a character $C$ of $\Gamma$}
\KwOut{An optimal solution to the PPS on $N$ and $C$}
Let $Q$ be an empty queue\;
$W\leftarrow\infty$\;
\For{$v\in V(N)$, $s\in S_p$}{
	Set counters $n_{v,s}^L$, $n_{v,s}^U$ and pseudo-cost scores $\Delta_{v,s}^L$, $\Delta_{v,s}^U$ equal to $0$\;
}
Let $(a,c)$ be an optimal solution to the LP relaxation of Formulation~\ref{form::bpps} for $N$ and $C$\;
$\lambda\leftarrow\text{ApproxPPS}(N,C)$\;
\eIf{$W(a,c)<W(\lambda)$}{
	Push Node$(a,c,\lambda,\emptyset)$ to $Q$\;
}{
	$\lambda^*\leftarrow\lambda$\;
	$W\leftarrow W(\lambda)$
}

\While{$Q\neq\emptyset$}{
	Pop Node$(a,c,\lambda,F)$ from $Q$\;
	\If{$W(a,c)<W$}{
	\If{$W(\lambda)<W$}{
		$\lambda^*\leftarrow\lambda$\;
		$W\leftarrow W(\lambda)$\;
	}
	\uIf{$(a,c)$ is binary}{
		$\lambda^*\leftarrow\lambda_a$\;
		$W\leftarrow W(\lambda_a)$\;
	}
	\ElseIf{$W(a,c)<W(\lambda)$}{
		\For{non-binary variables $a_v^s$}{
			For $G=F\cup\{(v,s)\}$, calculate LP solution $(a(G),c(G))$ and ApproxPPS solution $\lambda_G$\;
			$\text{nodes}(v,s)\leftarrow\text{nodes}(v,s)\,\cup\,\{\text{Node}(a(G),c(G),\lambda_G,G)\}$\;
			$\Delta_{v,s}^L\leftarrow\left(n_{v,s}^L\cdot\,\Delta_{v,s}^L+W(a(G),c(G))-W(a,c)\right)/(n_{v,s}^L+1)$\;
			$\Delta_{v,s}^U\leftarrow\left(n_{v,s}^U\cdot\,\Delta_{v,s}^U+W(\lambda)-W(\lambda_G)\right)/(n_{v,s}^U+1)$\;
			Increase $n_{v,s}^L$ and $n_{v,s}^U$ by one\;
		}
		For all $s\in S_p$, push elements in nodes$(v,s)$ with $\frac{1}{2(p+1)}\left(\sum_{t\in S_p}\Delta_{v,t}^L+\sum_{t\in S_p}\Delta_{v,t}^U\right)$ maximum to $Q$\;
	}
	}
}
\Return{$\lambda^*$}\;
\caption{exact solution algorithm for the PPS}\label{PPSexact}
\end{algorithm}

To get to an exact solution solution for the PPS, for a binary solution $(a,c)$ to Formulation~\ref{form::bpps} or its LP relaxation, we define a lineage function $\lambda_a$ on $N$ by 
\begin{align*}
\lambda_a(v)&=\left\{\argmax\left|\left\{s\,:\,s\in S_p,\,a_v^s=1\right\}\right|\right\}~&~&\forall\,v\in V(N).
\end{align*}
In addition, for a lineage function $\lambda$ on $N$ we write $W(\lambda)=\sum_{v\in V(N)}w_{\lambda}(v)$. Observe that $W(\lambda_a)$ denotes the objective function value of Formulation~\ref{form::bpps} for binary solution $(a,c)$. We also denote the same objective function value by $W(a,c)$ to not distinguish between $(a,c)$ being binary or not. We will employ strong branching~\cite{cook95} to solve the PPS, i.e., we compare an intermediate LP lower bound $W(a,b)$ with the LP lower bound obtained by branching on a non-binary variable value and subsequently choose as the branching variable the one with the largest improvement in lower bound $W(a,b)$. Strong branching is known to produce small branch-\&-bound search trees and it serves as a useful building block for more sophisticated branching strategies~\cite{aardal24}. To measure the improvement in $W(a,b)$ by strong branching, we denote an optimal solution to the LP relaxation of Formulation~\ref{form::bpps} under the additional constraints $a_v^s=1$ for some $v\in V(N)$, $s\in S_p$, encoded as a set $F$ of tuples $(v,s)$, by $(a(F),c(F))$. In addition, we write $\lambda_F$ for the lineage function returned by ApproxPPS$(N,C)$ under the additional constraints $s\in\lambda(v)$ for all $(v,s)\in F$. The latter measures the improvement in upper bounds which can contribute to pruning the search further. While we do not know any guarantees on the solution quality of ApproxPPS$(N,C)$ under such constraints the solutions to the PPS from our empirical study in the next section show that the additional constraints can lead to improvements in upper bounds which prune the search. Feasibility of ApproxPPS$(N,C)$ under the additional constraints is ensured by excluding assignments of $\lambda$ in Algorithms~\ref{SPSprocessN2} and~\ref{SPSapprox} which violate constraints $s\in\lambda(v)$. Since the disadvantage of strong branching is the high cost of solving many auxiliary linear programs, we combine our strong branching rule with pseudo-cost branching. Coupled with a capacity constraint on the number of strong branches this approach is also known as reliability branching~\cite{achterberg05}. Specifically, we re-use the differences $W(a(F),c(F))-W(\lambda)$ which we calculate for strong branching to determine pseudo-cost scores per branching variable based on running averages over differences $W(a(F),c(F))-W(\lambda)$. Preliminary computational experiments showed that the LP relaxation of Formulation~\ref{form::bpps} often contains half-integral values for many variables $a_v^s$. Hence, we aggregate pseudo-cost scores per vertex $v$ and do not differentiate between the different states $s\in S_p$ for branching. Thus, Algorithm~\ref{PPSexact} is a branch-\&-bound algorithm for the PPS on $N\in\mathcal{N}$ adapting both LP lower bounds and upper bounds obtained from ApproxPPS for pruning:
\begin{prop}
Algorithm~\ref{PPSexact} is correct.
\end{prop}
\begin{proof}
Let $N\in\mathcal{N}$ and let $C$ be a character of $\Gamma$. Algorithm~\ref{PPSexact} maintains a queue $Q$ of quadruples $(a,c,\lambda,F)$ in which $(a,c)$ and $\lambda$ provide lower bound $W(a,b)$ and upper bound $W(\lambda)$ on the optimal objective function value of the PPS on $N$ and $C$, respectively. The set $F$ encodes a collection of variables $a_v^s$ in $(a,c)$ such that $(a(F),c(F))$ contains only binary values for all $a_v^s$ appearing in $F$ and remains identical to $(a,c)$ in all variables (see line~23). The particular choice of binary values is independent of $s\in S_p$ (see line~28). Notice that elements from $F$ are never removed and that constraints $a_v^s=1$ generate new elements in $Q$ for each $s\in S_p$ independently. Due to constraints~\eqref{bpps::con3} and~\eqref{bpps::con5} each assignment constraint $a_v^s=1$ imposes upper bounds on $\sum_{s\in S_p}a_v^s$. Hence, eventually fixing $a_v^s=1$ for multiple $s\in S_p$ yields a binary assignment to variables $a_v^1,\dots,a_v^p$. This means, if we exclude all if-conditions involving quantities $W(a,c)$, $W(\lambda_a)$ or $W$ from Algorithm~\ref{PPSexact}, then the method does not exclude optimal integer solutions from the LP relaxation of Formulation~\ref{form::bpps}. Consequently an optimal solution to the PPS on $N$ and $C$ is found. All if-conditions prune this search process by appropriately comparing lower bounds $W(a,b)$ / $W(\lambda_a)$ and upper bounds $W(\lambda)$ with the intermediate minimal objective function value $W$ achieved by a found feasible solution to the PPS on $N$ and $C$.
\end{proof}

\section{Computational experiments}\label{sec:exp}
In this section we study the performance of the LP relaxation of Formulation~\ref{form::bpps}, algorithm ApproxPPS and Algorithm~\ref{PPSexact} on simulated data. Each simulated instance consists of a network $N\in\mathcal{N}$ and a character $C$ of $\Gamma$. Specifically, $N$ is constructed by first generating a phylogenetic tree of $\Gamma$ with simulator Ngesh (version 1.2.1) with default parameters~\cite{ngesh21}, and subsequently inserting edges by sampling pairs of vertices uniformly at random to meet given requirements on a fixed number of reticulation vertices and a fixed reticulation depth while keeping $N$ binary and tree-child. The characters $C$ are constructed by uniformly at random assigning states from $S_p$ to each taxon in $\Gamma$.

All the implementations described have been coded in Python version 3.12.13 by relying on FICO Xpress Optimizer libraries v47.01.01 for linear programming. All experiments were conducted on a MacBook Pro with M2 chip, 8-core CPU and 10-core GPU, 16 GB RAM and operation system macOS Tahoe (version 26.5.1) using Darwin kernel (version 25.5.0). The link to the online repository containing all implemented algorithms can be found at \url{https://github.com/mfrohn/ApproxPPS}.

\subsection{Empirical results}
\begin{table}[!t]
\normalsize
\begin{tabular}{cccrrrrrrrr}
\toprule
 \multicolumn{3}{c}{Data set} & \multicolumn{2}{c}{LP relaxation} & \multicolumn{2}{c}{ApproxPPS} &  \multicolumn{2}{c}{Exact} \\
 \multicolumn{3}{c}{\rule[1mm]{4cm}{0.2mm}}  & \multicolumn{2}{c}{\rule[1mm]{3cm}{0.2mm}} & \multicolumn{2}{c}{\rule[1mm]{3cm}{0.2mm}} & \multicolumn{2}{c}{\rule[1mm]{3cm}{0.2mm}}\\
Taxa & Reticulations & States & Time (sec) & Gap (\%) & Time (sec) & Factor & Time (sec) & Branches\\
\toprule
50 & 5 & 2 & 0.02 & 0.77 & 0.002 & 1.09 & 0.21 & 1\\
 & & 4 & 0.02 & 0.94 & 0.002 & 1.07 & 0.85 & 5\\
 & 20 & 2 & 0.02 & 5.84 & 0.003 & 1.41 & 6.48 & 11\\
 & & 4 & 0.02 & 3.05 & 0.004 & 1.28 & 12.15 & 16\\
100 & 10 & 2 & 0.02 & 0.22 & 0.004 & 1.08 & 0.49 & 1\\
& & 4 & 0.02 & 0.24 & 0.005 & 1.09 & 0.60 & 1\\
& 40 & 2 & 0.02 & 13.17 & 0.01 & 1.37 & 115.81 & 61\\
& & 4 & 0.04 & 1.72 & 0.01 & 1.21 & 114.57 & 44\\
500 & 50 & 2 & 0.09 & n.a. & 0.06 & n.a. & n.a. & n.a.\\ 
& 200 & 2 & 0.18 & n.a. & 0.58 & n.a. & n.a. & n.a.\\
1000 & 100 & 2 & 0.30 & n.a. & 0.41 & n.a. & n.a. & n.a.\\
 & 400 & 2 & 0.63 & n.a. & 4.17 & n.a. & n.a. & n.a.\\
 \bottomrule
\end{tabular}
\caption{Results for solving the PPS on simulated binary, semi-simplex, tree-child networks with the number of leaves, numbers of reticulations and cardinality of $S_p$ as parameters (see the first three columns). Columns 4 and 5 refer to the LP relaxation of Formulation~\ref{form::bpps}, columns 6 and 7 refer to the algorithm ApproxPPS, and columns 8 and 9 refer to Algorithm~\ref{PPSexact}. All reported numbers are averages over the results from the respective data sets, containing 25 networks and characters each. For solutions $(a,c)$ and $\lambda^*$ to the LP relaxation and Algorithm~\ref{PPSexact}, respectively, the optimality gap is given by the relative difference $(W(\lambda^*)-W(a,c))/W(\lambda^*)$. The factor for ApproxPPS is its empirical approximation factor. Branches depicts the number of coupled strong and pseudo-cost branching decisions (line~28 in Algorithm~\ref{PPSexact}).}
\label{experiments1}
\end{table}

First, in Table~\ref{experiments1} we compare the average quality and runtime of the LP relaxation of Formulation~\ref{form::bpps}, ApproxPPS and Algorithm~\ref{PPSexact} on binary, semi-simplex, tree-child networks. Recall that Algorithm~\ref{PPSexact} employs the other two algorithms as subroutines both in their original form at the root of the branch-\&-bound scheme and adaptively throughout the branching process. Up to networks with 100 taxa and 40 reticulation vertices the average runtimes do not exceed two minutes. However, outliers in the data (documented in the online repository) require runtimes of over one hour individually. This effect is more pronounced for even larger networks which makes the study of average quality and runtime of Algorithm~\ref{PPSexact} intractable with a one-hour threshold, i.e., a 25-hour threshold for each data set. Up to 1000 taxa we observe that the runtime of the LP relaxation scales significantly better than ApproxPPS. This phenomenon is supported by the simplicity of the constraints in Formulation~\ref{form::bpps} and its small size as well as the significant computational overhead of auxiliary calculations in the propagation scheme ApproxPPS. While an approximation factor of at most two for ApproxPPS is attained for individual inputs, on average the quality of the approximation is significantly better. Moreover, both the empirical approximation factor and the optimality gap of the LP relaxation are negatively affected by increases in the number of reticulation vertices. In addition, both measures indicate that the underlying methods produce higher quality solutions for more than two character states. Hence, an average case analysis might align with these observations on data simulated by a birth-death process with uniformly random edge insertions.

\definecolor{myTeal}{HTML}{5C98A0}
\definecolor{myGold}{HTML}{FFCC00}
\begin{figure}[pos=!t,align=\centering]
\begin{tikzpicture}
	\pgfplotstableread[col sep=comma]{data1.csv}\csvdata
	\pgfplotstabletranspose\datatransposed{\csvdata} 
	\begin{axis}[
		width=0.3*\textwidth,
		height=6cm,
		boxplot/draw direction = y,
		x axis line style = {opacity=0},
		axis x line* = bottom,
		axis y line = left,
		enlarge y limits,
		ymajorgrids,
		xtick = {1, 2, 3, 4},
		xticklabel style = {align=center, font=\small},
		xticklabels = {(1,2), (1,4), (5,2), (5,4)},
		xtick style = {draw=none}, 
		ylabel = {Runtime (sec)},
		ymax = 7,
		ytick = {0.5,2,4,7}
	]
		
		\addplot+[boxplot, fill=myTeal!70, draw=black, mark=*, mark options={fill=myTeal!70, draw=black, thin}] table[y index=1] {\datatransposed};
		\addplot+[boxplot, fill=myGold!70, draw=black, mark=*, mark options={fill=myGold!70, draw=black, thin}] table[y index=2] {\datatransposed};
		\addplot+[boxplot, fill=myTeal!70, draw=black, mark=*, mark options={fill=myTeal!70, draw=black, thin}] table[y index=3] {\datatransposed};
		\addplot+[boxplot, fill=myGold!70, draw=black, mark=*, mark options={fill=myGold!70, draw=black, thin}] table[y index=4] {\datatransposed};
				
	\end{axis}
\end{tikzpicture}~~~~
\begin{tikzpicture}
	\pgfplotstableread[col sep=comma]{data2.csv}\csvdata
	\pgfplotstabletranspose\datatransposed{\csvdata} 
	\begin{axis}[
		width=0.3*\textwidth,
		height=6cm,
		boxplot/draw direction = y,
		x axis line style = {opacity=0},
		axis x line* = bottom,
		axis y line = left,
		enlarge y limits,
		ymajorgrids,
		xtick = {1, 2, 3, 4},
		xticklabel style = {align=center, font=\small},
		xticklabels = {(1,2),(1,4),(5,2),(5,4)},
		xtick style = {draw=none}, 
		ylabel = {Runtime (sec)},
		ymax = 220,
		ytick = {10,50,100,200}
	]
		\addplot+[boxplot, fill=myTeal!70, draw=black, mark=*, mark options={fill=myTeal!70, draw=black, thin}] table[y index=1] {\datatransposed};
		\addplot+[boxplot, fill=myGold!70, draw=black, mark=*, mark options={fill=myGold!70, draw=black, thin}] table[y index=2] {\datatransposed};
		\addplot+[boxplot, fill=myTeal!70, draw=black, mark=*, mark options={fill=myTeal!70, draw=black, thin}] table[y index=3] {\datatransposed};
		\addplot+[boxplot, fill=myGold!70, draw=black, mark=*, mark options={fill=myGold!70, draw=black, thin}] table[y index=4] {\datatransposed};
		
	\end{axis}
\end{tikzpicture}~\\~\\
\begin{tikzpicture}
	\pgfplotstableread[col sep=comma]{data3.csv}\csvdata
	\pgfplotstabletranspose\datatransposed{\csvdata} 
	\begin{axis}[
		width=0.3*\textwidth,
		height=6cm,
		boxplot/draw direction = y,
		x axis line style = {opacity=0},
		axis x line* = bottom,
		axis y line = left,
		enlarge y limits,
		ymajorgrids,
		xtick = {1, 2, 3, 4},
		xticklabel style = {align=center, font=\small},
		xticklabels = {(1,2),(1,4),(5,2),(5,4)},
		xtick style = {draw=none}, 
		ylabel = {Runtime (sec)},
		ymax = 18,
		ytick = {1,5,10,15}
	]
		\addplot+[boxplot, fill=myTeal!70, draw=black, mark=*, mark options={fill=myTeal!70, draw=black, thin}] table[y index=1] {\datatransposed};
		\addplot+[boxplot, fill=myGold!70, draw=black, mark=*, mark options={fill=myGold!70, draw=black, thin}] table[y index=2] {\datatransposed};
		\addplot+[boxplot, fill=myTeal!70, draw=black, mark=*, mark options={fill=myTeal!70, draw=black, thin}] table[y index=3] {\datatransposed};
		\addplot+[boxplot, fill=myGold!70, draw=black, mark=*, mark options={fill=myGold!70, draw=black, thin}] table[y index=4] {\datatransposed};
		
	\end{axis}
\end{tikzpicture}~~~~
\begin{tikzpicture}
	\pgfplotstableread[col sep=comma]{data4.csv}\csvdata
	\pgfplotstabletranspose\datatransposed{\csvdata} 
	\begin{axis}[
		width=0.3*\textwidth,
		height=6cm,
		boxplot/draw direction = y,
		x axis line style = {opacity=0},
		axis x line* = bottom,
		axis y line = left,
		enlarge y limits,
		ymajorgrids,
		xtick = {1, 2, 3, 4},
		xticklabel style = {align=center, font=\small},
		xticklabels = {(1,2),(1,4),(5,2),(5,4)},
		xtick style = {draw=none}, 
		ylabel = {Runtime (sec)},
		ymax = 1800,
		ytick = {100,500,1000,1500}
	]
		\addplot+[boxplot, fill=myTeal!70, draw=black, mark=*, mark options={fill=myTeal!70, draw=black, thin}] table[y index=1] {\datatransposed};
		\addplot+[boxplot, fill=myGold!70, draw=black, mark=*, mark options={fill=myGold!70, draw=black, thin}] table[y index=2] {\datatransposed};
		\addplot+[boxplot, fill=myTeal!70, draw=black, mark=*, mark options={fill=myTeal!70, draw=black, thin}] table[y index=3] {\datatransposed};
		\addplot+[boxplot, fill=myGold!70, draw=black, mark=*, mark options={fill=myGold!70, draw=black, thin}] table[y index=4] {\datatransposed};
		
	\end{axis}
\end{tikzpicture}
\caption{A boxplot for the runtime of Algorithm~\ref{PPSexact} on 16 data sets each containing 25 simulated binary, tree-child networks. The tuples $(s,t)$ on the x-axis parameterize the networks to have reticulation depth at most $s$ and number of character states $t$. Furthermore, the figure on the top-left / top-right / bottom-left / bottom-right requires 50/50/100/100 taxa and 5/20/10/40 reticulations, respectively.}\label{experiments2}
\end{figure}

Next, consider the performance of Algorithm~\ref{PPSexact} illustrated in Figure~\ref{experiments2}. Observe that plot parameters of the form $(1,t)$ correspond to the semi-simplex networks in Table~\ref{experiments1}. Here, in addition to average runtimes, outliers are shown. We can see that there are no clear differences were observed between runtimes for fixed numbers of reticulation vertices. Moreover, no parameter setting is always better than another. However, increasing the reticulation depth to at most five on binary characters leads to reliably low runtimes for Algorithm~\ref{PPSexact} with the exception of one outlier for 100 taxa with 10 reticulation vertices. Due to the randomness of the simulated data and the prevalence of statistical outliers in it, further study is warranted to explain this behaviour. We can also observe that moving from two to four character states slows Algorithm~\ref{PPSexact} down in the worst case. This can be explained by the averaging effect in line~28 of Algorithm~\ref{PPSexact} as pairwise distinct character states are treated as equally valuable for branching. Hence, extending Algorithm~\ref{PPSexact} to learn which character states are more beneficial than others to converge to an optimal solution could ameliorate the runtime increase for non-binary characters. Finally, comparing the boxplots between the left and right figures in Figure~\ref{experiments2} we can see that the increase in the ratio of reticulation vertices in the respective networks from 10\% of leaves to 40\% of leaves can increase the runtime by two magnitudes. This suggests that the complexity of the PPS is more closely linked to the density of reticulation vertices in the network beyond the reticulation depth measure. These empirical trends for the reticulation density and depth motivate a further parameterized analysis of the PPS in future research.

\section{Conclusion and future research}\label{sec:conclusion}
We have shown how to approximate the PPS on binary, semi-simplex, tree-child phylogenetic networks and arbitrary characters with factor two in polynomial time. Furthermore, we have introduced an integer programming model for the PPS on binary, tree-child phylogenetic networks to solve the PPS exactly on simulated instances of up to 100 taxa. To the best of our knowledge this article contains the first empirical study of the PPS. From a practical point of view this opens up multiple research questions for future works: first, this work does not delve into the biological assumptions on the input data which could benefit a further structured and larger scale computational study of our new algorithms. In particular, our algorithms might help to improve methods like NetRAX~\cite{lutteropp22} who use the parsimony score of a phylogenetic network as part of an inference of evolutionary histories by extending the support to parental parsimony, i.e., allow for allopolyploidy and incomplete lineage sorting in the inference process. Secondly, our experiments indicate that the integrality gap of Formulation~\ref{form::bpps} is relatively small. Moreover, Formulation~\ref{form::bpps} exhibits structural similarities to the integer programming model for the SPS by Fischer et al.~\cite{fischer15}. While a proof for the integrality gap of both models is desirable, it also raises the practical question whether the integer programming model for the SPS of Schmidt \& Raphael~\cite{schmidt25}, which shows better performance than the model by Fischer et al.~\cite{fischer15} (see Figure 2b in \cite{schmidt25}), can be extended to derive more efficiently tractable lower bounds for the PPS to improve Algorithm~\ref{PPSexact}. Finally, our construction of Algorithm~\ref{PPSexact} has hinted at the use of reliability branching. A more detailed study of such machine learning techniques to solve the PPS exactly would benefit scalability. Specifically, the effects of our adaptive use of lower and upper bounds to solve the PPS on learned pseudo-costs can be studied to limit redundancies in the branch-\&-bound procedure.

Next to the theoretical questions that accompany empirical studies of the PPS we also want to point out some broader generalizations of the PPS as an operations research problem in itself. On the one hand, the input graph and vertex labels can be more general than binary, tree-child phylogenetic networks with leaf labels. Both orchard and tree-based phylogenetic networks~\cite{kong1} generalize the tree-child property, the underlying graph might be considered as regular instead of binary and labels can be extended to some internal vertices. We have already introduced the latter point indirectly in our adaptive implementation of ApproxPPS by requiring the input to satisfy $s\in\lambda(v)$ for some states $s\in S_p$ and any lineage function $\lambda$ on the corresponding network. Furthermore, the use of the Hamming distance as a binary/uniform metric to calculate the parental parsimony score can be generalized, too. These variations of the PPS are reminiscent of the NP-hard metric labeling problem~\cite{tardos02} which, given a weighted graph, a set of vertex labels and a metric distance function asks for an assignment of labels to vertices and the separation of vertices with pairwise distinct labels, minimizing both an assignment and separation cost. For the uniform metric the metric labeling problem can be approximated with factor two~\cite{tardos02}, a version similar to the PPS which excludes the label assignment from the decision making but in addition asks for an embedding of a most parsimonious parentally displayed tree within the input graph instead of minimizing the total separation cost of the graph. Studying the connection between the PPS and the metric labeling problem could therefore unveil structural similarities between problems in phylogenetics and application of the metric labeling problem such as image segmentation and computer vision.

Finally, we mention some additional combinatorial aspects of the PPS when viewed as an optimization problem in itself. First, we have studied the parsimony score based on a single given character of a set of taxa. If instead we consider molecular data given by an alignment of molecular sequences~\cite{FelsenBook}, then we ask to find a parentally displayed tree based on multiple characters. This generalization of the PPS can be relaxed into multiple instances of the PPS if we ask for a parentally displayed tree for each character independently. Otherwise, further research is warranted as we have no proof that Proposition~\ref{SPSapprox} still holds, i.e., the work of~\citet{frohn25} (which we have only reviewed) needs to be generalized. The same question arises when we restrict our attention to normal networks~\cite{kong1} instead of tree-child network, i.e., we consider tree-child networks without short-cuts. One would expect that this additional topological restriction strengthens the approximation guarantee of Proposition~\ref{SPSapprox}. However it is unclear how the primal-dual framework of~\citet{frohn25} which relies on local optimality conditions can leverage a global network property like no short-cuts. This difficulty is similar to the study of the PPS on the aforementioned orchard networks for which, to the best of our knowledge, no characterization by forbidden subgraphs is known. Lastly, we have seen a clear difference in the combinatorial aspects of the PPS and SPS in Figure~\ref{intro} which we have excluded by restricting our combinatorial study to semi-simplex networks. Lifting this restriction requires new ideas as our notion of a reticulation graph extension is not well-defined for networks which are not semi-simplex.

\section*{Acknowledgments}
The author thanks Olaf Steenstra and Steven Kelk for computational support and helpful discussions. This work was supported by grant OCENW.M.21.306 of the Dutch Research Council (NWO).

\bibliographystyle{cas-model2-names}
\bibliography{biblioteca}
\end{document}